\documentclass[12pt,a4paper]{article}
\usepackage{amssymb}
\usepackage{amsmath}
\usepackage{amsthm}

\newtheorem{theorem}{Theorem}[section]
\newtheorem{definition}{Definition}[section]
\newtheorem{proposition}{Proposition}[section]
\newtheorem{example}{Example}[section]
\newtheorem{remark}{Remark}[section]

\numberwithin{equation}{section}

\title{\textbf{Multitime controlled linear PDE systems}}

\author{\textbf{Cristian Ghiu and Constantin Udri\c ste}}

\date{}

\begin{document}

\maketitle

\begin{center}
University Politehnica of Bucharest,
Faculty of Applied Sciences, Department of Mathematics II,
Splaiul Independen\c tei 313, 060042 Bucharest,
Romania, e-mail: crisghiu@yahoo.com
\end{center}

\begin{center}
University Politehnica of Bucharest, Faculty of Applied Sciences,
Department of Mathematics-Informatics I,
Splaiul Independen\c tei 313, 060042 Bucharest, Romania,
e-mail: udriste@mathem.pub.ro, anet.udri@yahoo.com
\end{center}

\begin{abstract}
We derive new results regarding the controllability and the reachability of multitime controlled linear PDE systems of first order.
These systems describe some important multitime evolution in engineering, economics and biology. Some of them come from
evolution PDEs of superior order.
The original results include a refinement and a supplement of multitime optimal control theory,
developed in some recent papers by the second author. They refer to the complete integrability conditions,
conditions for the existence of solutions, path independent curvilinear integrals, the multitime fundamental matrix,
multitime adjoint Cauchy problems, control space, controllability and  reachability of phases,
controllability gramian,  reachability gramian, controllability matrix, counter-examples and commentaries.
\end{abstract}

AMS Subject Classification: 93B05, 82C70, 82B40; 49J20, 49J45, 68U10.

Keywords: multitime controllability, multitime reachability, multitime controllability gramian.

\section{Introduction}

Here a controlled system is a dynamic multitime linear PDE system on which one can act by using appropriate controls.
Among the most common problems that appear when studying such systems are multitime
controllability problem and multitime reachability problem.

The multitime controllability refers to pairs of states that can be moved from the first one to the second one
and the multitime reachability operates on the reverse order of states. Of course, the order of states is given by the product
order (partial order) on multitime source space. The study of controllability of dynamical systems represented by normal PDEs
starts in the papers \cite{12} -- \cite{23}, \cite{3} (multitime maximum principle),
\cite{7}, \cite{8}, \cite{24} (maximum principle in the context of weak derivatives),
\cite{4} (numerical methods for robust control),  \cite{10}, \cite{11} (reachability
of hybrid systems).
Though many of situations are rather well understood, there are still quite challenging open problems due
to the fact that the product order relation on multitime space is not total.

This paper deals with control theory for systems governed by multitime linear PDE systems ($m$-flows).
Section 2 presents a new and complete framework for the multitime nonautonomous linear PDE systems of first order.
Section 3 contains original results about controllability and reachability of the controlled multitime nonautonomous linear PDE systems of first order.
The controllability and the reachability of multitime autonomous linear PDE systems of first order is analyzed in Section 4. The comments (Section 5)
show that in some other situations can occur strange mathematical phenomena due to the discontinuity of controls in multitime evolutions.

\section{Nonautonomous linear PDE system of \\first order}

We start with some mathematical ingredients related to evolution PDEs ($m$-flows). Let $t=(t^1,\ldots,t^m)\in \mathbb{R}^m$, called {\it multitime},
$x=(x^1,\ldots,x^n)^{\top} \in \mathbb{R}^n=\mathcal{M}_{n,1}(\mathbb{R})$, and $G\subseteq \mathbb{R}^m \times \mathbb{R}^n$
be an open subset. We consider the evolution PDE system
\begin{equation}
\label{aa1}
\frac{\partial x}{\partial t^{\alpha}}(t) = X_{\alpha}(t,x(t)), \quad \forall \alpha=\overline{1,m},
\end{equation}
where  $X_{\alpha}:G \to \mathbb{R}^n=\mathcal{M}_{n,1}(\mathbb{R})$, $X_{\alpha}=(X^1_{\alpha},\ldots,X^m_{\alpha})$.

\begin{definition}
\label{definitia2.1} {\it The PDE system $(\ref{aa1})$ is called completely integrable if $\forall (t_0,x_0)\in G$,
$\exists D_0\subseteq \mathbb{R}^m$, $D_0$ open with
$t_0\in D_0$ and $\exists x:D_0\to \mathbb{R}^n$, $x$ differentiable, such that $(t,x(t))\in G$, $\forall t\in D_0$, $x$
verifies $(\ref{aa1})$ and $x(t_0)=x_0$.}
\end{definition}

The following Theorems, \ref{teorema2.1} to \ref{teorema2.4}, represent new versions of some
well-known results \cite{5}, \cite{12} -- \cite{23}.

\begin{theorem}
\label{teorema2.1} {\it Suppose the components $X_\alpha$ are of class
$\mathcal{C}^1$, $\forall \alpha=\overline{1,m}$.

i) Any solution of the PDE system} $(\ref{aa1})$ {\it is of class $\mathcal{C}^2$.

ii) If the PDE system} $(\ref{aa1})$ {\it is completely integrable, then
}
\begin{equation}
\label{aa2}
\begin{split}
\frac{\partial X_{\alpha}}{\partial t^{\beta}}(t,x)
+
X^j_{\beta}(t,x)
\frac{\partial X_{\alpha}}{\partial x^{j}}(t,x)
&=
\frac{\partial X_{\beta}}{\partial t^{\alpha}}(t,x)
+
X^j_{\alpha}(t,x)
\frac{\partial X_{\beta}}{\partial x^{j}}(t,x),\\
\forall (t,x)\in G
&\mbox{,\,\,\,}
\forall \alpha,\beta=\overline{1,m}
\end{split}
\end{equation}
{\it or in matrix notations}
\begin{equation}
\label{aa3}
\begin{split}
\frac{\partial X_{\alpha}}{\partial t^{\beta}}
+
\left(
\frac{\partial X_{\alpha}}{\partial x^1}
\,
\cdots
\,
\frac{\partial X_{\alpha}}{\partial x^n}
\right)
X_{\beta}
&=
\frac{\partial X_{\beta}}{\partial t^{\alpha}}(t,x)
+
\left(
\frac{\partial X_{\beta}}{\partial x^1}
\,
\cdots
\,
\frac{\partial X_{\beta}}{\partial x^n}
\right)
X_{\alpha},\\
\forall (t,x)\in G
&\mbox{,\,\,\,}
\forall \alpha,\beta=\overline{1,m}.
\end{split}
\end{equation}
{\it The relations $(\ref{aa2})$ or $(\ref{aa3})$ are called
the complete integrability conditions.}
\end{theorem}

\begin{theorem}[Frobenius]
\label{teorema2.2}{\it Let
$G\subseteq \mathbb{R}^m \times \mathbb{R}^n$
be an open subset and}
$$X_{\alpha}:G \to \mathbb{R}^n=\mathcal{M}_{n,1}(\mathbb{R}),\,\, X_\alpha\, \,of\, class\,\, \mathcal{C}^1, \forall \alpha=\overline{1,m}.$$

$a)$ {\it If the conditions $(\ref{aa2})$ are satisfied}, {\it then
the PDE system} $(\ref{aa1})$ {\it is completely integrable.}

$b)$ {\it Let $D\subseteq \mathbb{R}^m$ be an open and convex subset and
$G=D\times \mathbb{R}^n$. Suppose that the following condition is fulfilled:
$\exists R\geq0$ {\it and there exist the continuous functions}\,\,$\varphi,\psi: D\to [0,\infty)$ {\it such that}
\begin{equation}
\label{aa4}
||X_{\alpha}(t,x)\| \leq \varphi(t)\|x\|+\psi(t), \,\, \forall t\in D, \,\, \forall x\in \mathbb{R}^n,\,\,\|x\|\geq R, \,\,\forall \alpha=\overline{1,m}.
\end{equation}  }
{\it (For example, if the PDE} $(\ref{aa1})$ {\it is linear, then
the conditions} $(\ref{aa4})$ {\it are satisfied). }

{\it If the complete integrability conditions} $(\ref{aa2})$ {\it are satisfied, then:
$\forall (t_0,x_0)\in D \times \mathbb{R}^n$,
$\exists x: D\to \mathbb{R}^n$, $x$ of class $\mathcal{C}^2$,
solution of the PDE system} $(\ref{aa1})$ {\it and $x(t_0)=x_0$.  }
\end{theorem}

\begin{theorem}
\label{teorema2.3}{\it Let
$G\subseteq \mathbb{R}^m \times \mathbb{R}^n$ be an open subset and
$$X_{\alpha}:G \to \mathbb{R}^n=\mathcal{M}_{n,1}(\mathbb{R}), X_\alpha\,\,\hbox{of class}\,\, \mathcal{C}^1, \forall \alpha=\overline{1,m}.$$}
{\it Let $D_1,D_2\subseteq \mathbb{R}^m$ be open subsets and $y:D_1\to \mathbb{R}^n$,
$z:D_2\to \mathbb{R}^n$ be solutions of the PDE system $(\ref{aa1})$. If $D_1\cap D_2$ is connected and there exists
$t_0\in D_1\cap D_2$ such that $y(t_0)=z(t_0)$, then} $y(t)=z(t), \, \forall t\in D_1\cap D_2$.
\end{theorem}

\begin{definition}
\label{definitia2.2} {\it Let
$D\subseteq \mathbb{R}^m $ be an open subset and
$P_{\alpha}:D\to \mathcal{M}_{n,k}(\mathbb{R})$ be functions of class
$\mathcal{C}^1$. We say that the curvilinear integral
$\displaystyle
\int\limits_{\gamma}
P_{\alpha}(t)
\mbox{d}t^{\alpha}
$
is path independent (on $D$), if for any two points $t_0,t_1\in D$ and any two piecewise $\mathcal{C}^1$ curves
$\eta: [a,b]\to D$, $\lambda : [c,d] \to D$, with $\eta(a)=\lambda(c)=t_0$ and
$\eta(b)=\lambda(d)=t_1$, we have}
$$
\int\limits_{\eta}
P_{\alpha}(t)
\mbox{d}t^{\alpha}
=
\int\limits_{\lambda}
P_{\alpha}(t)
\mbox{d}t^{\alpha}.
$$
\end{definition}

\begin{theorem}
\label{teorema2.4}{\it Let $D\subseteq \mathbb{R}^m $ be an open
subset and $P_{\alpha}:D\to \mathcal{M}_{n,k}(\mathbb{R})$ be
$\mathcal{C}^1$ functions, $\forall \alpha=\overline{1,m}$. If $D$
is a convex set (sufficiently, connected and simply connected), then
the following statements are equivalent:

\vspace{0.2cm}
i) $\displaystyle
\frac{\partial P_{\alpha}}{\partial t^{\beta}}(t)
=
\frac{\partial P_{\beta}}{\partial t^{\alpha}}(t),
\quad
\forall t\in D,
\quad
\forall \alpha=\overline{1,m}.$

\vspace{0.2cm}
ii) $\exists \xi :D\to \mathcal{M}_{n,k}(\mathbb{R})$
solution of the PDE system
$$
\frac{\partial \xi}{\partial t^{\alpha}}(t)
=P_{\alpha}(t),
\quad
\forall t\in D,
\quad
\forall \alpha=\overline{1,m}.
$$

iii) The curvilinear integral
$$
\int\limits_{\gamma}
P_{\alpha}(t)
\mbox{d}t^{\alpha}
$$
is path independent on the set $D$.}
\end{theorem}

In the conditions $i)$ - $iii)$, we have:

$a)$ If $\xi$ is a solution of the PDE system of $ii)$, and $\gamma: [a,b]\to D$ is a piecewise $\mathcal{C}^1$ curve, then
$$
\int\limits_{\gamma}
P_{\alpha}(t)
\mbox{d}t^{\alpha}
=
\xi(\gamma(b))- \xi(\gamma(a)).
$$

$b)$ Let $t_0\in D$ be a fixed point. For $t\in D$, let
$\gamma_{t_0,t}: [a,b]\to D$ be a piecewise $\mathcal{C}^1$ curve, from
$t_0$ to $t$. The primitive
$$
\xi:D\to \mathcal{M}_{n,k}(\mathbb{R}),\,\,\xi(t)= \int\limits_{\gamma_{t_0,t}}P_{\alpha}(s)\mbox{d}s^{\alpha}
$$
is a solution of the PDE system of $ii)$.

For example, if $D$ is a star-shaped set with respect to $t_0$, then the primitive $\xi$ can be written alternatively
$$
\xi(t)
=
\int\limits_{0}^{1}
(t^{\alpha}-t_0^{\alpha})P_{\alpha}((1-\tau)t_0+\tau t)
\, \mbox{d}\tau.
$$

In case that the PDE system (\ref{aa1}) is linear, i.e.,
\begin{equation}
\label{aa5}
\frac{\partial x}{\partial t^{\alpha}}
=
M_{\alpha}(t)x+F_{\alpha}(t),
\quad
\forall \alpha=\overline{1,m},
\end{equation}
with $M_{\alpha}:D\to \mathcal{M}_n(\mathbb{R})$ and
$F_{\alpha}:D\to \mathbb{R}^n = \mathcal{M}_{n,1}(\mathbb{R})$
of class $\mathcal{C}^1$, the complete integrability conditions become
$$
\frac{\partial M_{\alpha}}{\partial t^{\beta}}(t)x
+
\frac{\partial F_{\alpha}}{\partial t^{\beta}}(t)
+
M_{\alpha}(t)
(
M_{\beta}(t)x+F_{\beta}(t)
)
$$
$$
= \frac{\partial M_{\beta}}{\partial t^{\alpha}}(t)x
+
\frac{\partial F_{\beta}}{\partial t^{\alpha}}(t)
+
M_{\beta}(t)
(
M_{\alpha}(t)x+F_{\alpha}(t)),
$$
$$\forall t\in D, \,\, \forall x\in \mathbb{R}^n, \,\, \forall \alpha,\beta=\overline{1,m},$$
which is equivalent to
\begin{equation}
\label{aa6}
\begin{split}
\frac{\partial M_{\alpha}}{\partial t^{\beta}}(t)
+
M_{\alpha}(t)M_{\beta}(t)
=
\frac{\partial M_{\beta}}{\partial t^{\alpha}}(t)
+
M_{\beta}(t)M_{\alpha}(t)
\end{split}
\end{equation}
\begin{equation}
\label{aa7}
M_{\alpha}(t)F_{\beta}(t)
+
\frac{\partial F_{\alpha}}{\partial t^{\beta}}(t)
=
M_{\beta}(t)F_{\alpha}(t)
+
\frac{\partial F_{\beta}}{\partial t^{\alpha}}(t).
\end{equation}

We have obtained the following result:

\begin{theorem}
\label{teorema2.5} {\it Let $D\subseteq \mathbb{R}^m$ be an open and convex subset, let
$M_{\alpha}:D \to \mathcal{M}_{n}(\mathbb{R})$ be $\mathcal{C}^1$ matrix functions,
$\forall \alpha=\overline{1,m}$ and let
$F_{\alpha}:D \to\mathbb{R}^n=\mathcal{M}_{n,1}(\mathbb{R})$ be $\mathcal{C}^1$ vector functions, $\forall \alpha=\overline{1,m}$.
Suppose the relations} $(\ref{aa6})$, $(\ref{aa7})$ {\it are true}, $\forall t\in D, \,\,
\forall \alpha,\beta=\overline{1,m}$. {\it Then the problem}
\begin{equation*}
\label{aa8}
\frac{\partial x}{\partial t^{\alpha}}
=
M_{\alpha}(t)x+F_{\alpha}(t), \quad \forall \alpha=\overline{1,m},
\end{equation*}
\begin{equation*}
\label{aa9}
x(t_0)=x_0
\end{equation*}
{\it has a unique solution $x:D\to \mathbb{R}^n$. This solution is of class $\mathcal{C}^2$. }
\end{theorem}

Further, everywhere, $D$ will be an open and convex subset of $\mathbb{R}^m$, and
$M_{\alpha}:D \to \mathcal{M}_{n}(\mathbb{R})$, \,\,$\forall \alpha=\overline{1,m}$, are matrix functions of class $\mathcal{C}^1$,
which verifies the relations (\ref{aa6}), $\forall t\in D, \,\, \forall \alpha,\beta=\overline{1,m}$.

There exists a unique matrix solution
$$\chi(\,\cdot\, , t_0):D\to \mathcal{M}_n( \mathbb{R})$$
of the problem
\begin{equation}
\label{aa10}
\frac{\partial X}{\partial t^{\alpha}}
=
M_{\alpha}(t)X,
\quad
\forall \alpha=\overline{1,m}
\end{equation}
\begin{equation*}
X(t_0)=I_n.
\end{equation*}
(For those $n$ problems equivalent to
the matrix problem, we apply the Theorem \ref{teorema2.5}).

\begin{definition}
\label{definitia2.3} {\it The matrix function}
$$\chi(\,\cdot\, , \cdot\,):D\times D \to \mathcal{M}_n( \mathbb{R})$$
{\it is called the fundamental matrix.}
\end{definition}

\begin{proposition}
\label{propozitia2.1} {\it The fundamental matrix has the following properties:

$a)$ $\chi(t,t_0)\chi(t_0,t_1)=\chi(t,t_1)$,\,\,
$\forall t_0, t_1, t\in D$,

$b)$ $\chi(t_0,t_0)=I_n$, \,
$\forall t_0\in D$,

$c)$ $\chi(t,t_0)^{-1}=\chi(t_0,t)$, \,
$\forall t_0, t\in D$.

$d)$ $\displaystyle
\frac{\partial}{\partial t^{\alpha}}
(\chi(t_0,t))=
-\chi(t_0,t)M_{\alpha}(t),
\quad
\forall t\in D,
\quad
\forall \alpha
$. }
\end{proposition}

\begin{proof}
$a)$ If $Y(t)=\chi(t,t_0)\chi(t_0,t_1)$, then
$$
\frac{\partial Y}{\partial t^{\alpha}}=
\frac{\partial }{\partial t^{\alpha}}
(\chi(t,t_0))\chi(t_0,t_1)=
M_{\alpha}(t)\chi(t,t_0)\chi(t_0,t_1)=
M_{\alpha}(t)Y;
$$
$$
Y(t_0)=\chi(t_0,t_0)\chi(t_0,t_1)=
I_n\chi(t_0,t_1)=
\chi(t_0,t_1).
$$
Hence $Y(t)$ and $\chi(t,t_1)$ are both solutions of the matrix PDE system
$$
\frac{\partial X}{\partial t^{\alpha}}
=
M_{\alpha}(t)X,
\quad \forall \alpha=\overline{1,m},
$$
which coincide for $t=t_0$. From uniqueness it follows
that $Y(t)=\chi(t,t_1)$, $\forall t$.\\
$b)$ Direct consequence of the definition of the function $\chi(t,t_0)$.\\
$c)$ It follows readily from a) and b). For a), we take $t_1=t$, etc.\\
$d)$ Differentiating the identity
$$
\chi(t,t_0)\chi(t_0,t)=I_n
$$
with respect to $t^\alpha$, we find
$$
\frac{\partial}{\partial t^{\alpha}}
(\chi(t,t_0))\chi(t_0,t)+
\chi(t,t_0)
\frac{\partial}{\partial t^{\alpha}}
(\chi(t_0,t))=0
$$
or
$$
M_{\alpha}(t)\chi(t,t_0)\chi(t_0,t)+
\chi(t,t_0)
\frac{\partial}{\partial t^{\alpha}}
(\chi(t_0,t))=0,
$$
i.e,
$$
\chi(t,t_0)
\frac{\partial}{\partial t^{\alpha}}
(\chi(t_0,t))=
-M_{\alpha}(t).
$$
Multiplying at the left-hand side by $\chi(t,t_0)^{-1}=\chi(t_0,t)$, we get
$$
\frac{\partial}{\partial t^{\alpha}}
(\chi(t_0,t))=
-\chi(t_0,t)M_{\alpha}(t).
$$
\end{proof}

\begin{proposition}
\label{propozitia2.2} {\it The Cauchy problem}
\begin{equation}
\label{aa12}
\frac{\partial x}{\partial t^{\alpha}}
=
M_{\alpha}(t)x,
\quad
\forall \alpha=\overline{1,m},
\end{equation}
\begin{equation*}
x(t_0)=x_0
\end{equation*}
{\it has the solution \,\,$x:D\to \mathbb{R}^n$, $x(t)=\chi(t,t_0)x_0$.}
\end{proposition}

\begin{definition}
\label{definitia2.4} {\it Let us consider
the PDE system} $(\ref{aa12})$.
{\it The homogeneous PDE system}
\begin{equation}
\frac{\partial y}{\partial t^{\alpha}}(t)
=
-M^\top_{\alpha}(t)y(t),
\quad
\forall \alpha=\overline{1,m}
\end{equation}
{\it is called the adjoint system}.
\end{definition}

The complete integrability conditions of the adjoint system are
$$
-\frac{\partial M^\top_{\alpha}}{\partial t^{\beta}}
+
M^\top_{\alpha}M^\top_{\beta}
=
-\frac{\partial M^\top_{\beta}}{\partial t^{\alpha}}
+
M^\top_{\beta}M^\top_{\alpha}
$$
or
$$
-\frac{\partial M_{\alpha}}{\partial t^{\beta}}
+
M_{\beta}M_{\alpha}
=
-\frac{\partial M_{\beta}}{\partial t^{\alpha}}
+
M_{\alpha}M_{\beta}
$$
$$
\frac{\partial M_{\alpha}}{\partial t^{\beta}}
+
M_{\alpha}M_{\beta}
=
\frac{\partial M_{\beta}}{\partial t^{\alpha}}
+
M_{\beta}M_{\alpha},
$$
i.e., identical to the relations (\ref{aa6}) of complete
integrability of the system (\ref{aa12}).

\begin{proposition}
\label{propozitia2.3} {\it $a)$ The matrix solution of the Cauchy problem
\begin{equation*}
\frac{\partial X}{\partial t^{\alpha}}
=
-M^\top_{\alpha}X,
\quad \forall \alpha=\overline{1,m}
\end{equation*}
\begin{equation*}
X(t_0)=I_n
\end{equation*}
is\, $\Phi(t,t_0)=\chi(t_0,t)^\top$.

$b)$ The solution of the adjoint Cauchy problem
\begin{equation*}
\frac{\partial \varphi}{\partial t^{\alpha}}
=
-M^\top_{\alpha}\varphi,
\quad \forall \alpha=\overline{1,m}
\end{equation*}
\begin{equation*}
\varphi(t_0)=\varphi_0
\end{equation*}
is}
$$\varphi(t)=
\Phi(t,t_0)\varphi_0=\chi(t_0,t)^\top\varphi_0.
$$
\end{proposition}
\begin{proof}
$a)$ We use the Proposition \ref{propozitia2.1}, $d)$, i.e.,
$$
\frac{\partial}{\partial t^{\alpha}}
(\chi(t_0,t))=
-\chi(t_0,t)M_{\alpha},
$$
which is equivalent to
$$
\frac{\partial}{\partial t^{\alpha}}
(\chi(t_0,t)^\top)=
-M_{\alpha}^\top\chi(t_0,t)^\top.
$$
$b)$ follows immediately from $a)$.
\end{proof}

\begin{theorem}
\label{teorema2.6} {\it In the conditions of
Theorem} $\ref{teorema2.5}$, {\it the solution of the Cauchy problem}
\begin{equation}
\label{aa15}
\frac{\partial x}{\partial t^{\alpha}}
=
M_{\alpha}(t)x+F_{\alpha}(t),
\quad
\forall \alpha=\overline{1,m},
\end{equation}
\begin{equation*}
x(t_0)=x_0
\end{equation*}
{\it is}
$$
x:D\to \mathbb{R}^n, \,\,x(t)=\chi(t,t_0)x_0
+
\int\limits_{\gamma_{t_0,t}}^{}
\chi(t,s)F_{\alpha}(s)
\mbox{d}s^{\alpha},
$$
{\it where $\gamma_{t_0,t}$ is a piecewise $\mathcal{C}^1$ curve,
included in $D$, covered from $t_0$ to $t$}.

{\it The curvilinear integral $\displaystyle \int\limits_{\gamma}^{}
\chi(t,s)F_{\alpha}(s)
\mbox{d}s^{\alpha}$ is path independent}.
\end{theorem}

\begin{proof} We show that the curvilinear integral is path
independent. According to the Theorem \ref{teorema2.4}, we must show that
$$
\frac{\partial }{\partial s^{\beta}}
\left(
\chi(t,s)F_{\alpha}(s)
\right)
=
\frac{\partial }{\partial s^{\alpha}}
\left(
\chi(t,s)F_{\beta}(s)
\right)
$$
or
$$
-\chi(t,s)M_{\beta}(s)F_{\alpha}(s)
+
\chi(t,s)
\frac{\partial F_{\alpha}}{\partial s^{\beta}}(s)
=
-\chi(t,s)M_{\alpha}(s)F_{\beta}(s)
+
\chi(t,s)
\frac{\partial F_{\beta}}{\partial s^{\alpha}}(s),
$$
and these are equivalent to the relations (\ref{aa7}).

Now we write the sheet $x(t)$ as
$$
x(t)=\chi(t,t_0)x_0
+
\chi(t,t_0)
\int\limits_{\gamma_{t_0,t}}^{}
\chi(t_0,s)F_{\alpha}(s)
\mbox{d}s^{\alpha}.
$$
According to the Theorem \ref{teorema2.4}, we get
$$
\frac{\partial} {\partial t^{\beta}}
\Big(
\int\limits_{\gamma_{t_0,t}}^{}
\chi(t_0,s)F_{\alpha}(s)
\mbox{d}s^{\alpha}
\Big)
=
\chi(t_0,t)F_{\beta}(t).
$$
It follows that
$$
\frac{\partial x} {\partial t^{\beta}}(t)
=
M_{\beta}(t)\chi(t,t_0)x_0
+
$$
$$
+
M_{\beta}(t)\chi(t,t_0)
\int\limits_{\gamma_{t_0,t}}^{}
\chi(t_0,s)F_{\alpha}(s)
\mbox{d}s^{\alpha}
+
\chi(t,t_0)\chi(t_0,t)F_{\beta}(t)
$$
$$
=
M_{\beta}(t)
\Big(
\chi(t,t_0)x_0
+
\int\limits_{\gamma_{t_0,t}}^{}
\chi(t,s)F_{\alpha}(s)
\mbox{d}s^{\alpha}
\Big)
+
F_{\beta}(t)
=
M_{\beta}(t)x(t)
+
F_{\beta}(t).
$$
One verifies easily the initial condition $x(t_0)=x_0$.
\end{proof}

\section{Controlled nonautonomous linear PDE \\system of first order}

Our main results include generalizations to multitime case of
the single-time control (see, for example, \cite{2}, \cite{9}) in the vision
of Lawrence C. Evans, Lev S. Pontryagin. They are complementary
to the results in \cite{3}, \cite{4}, \cite{7}, \cite{8}, \cite{10} -- \cite{23}.
Related topics can be found in the papers \cite{1}, \cite{6}.

Let $D\subseteq \mathbb{R}^m$ be an open and convex subset, let
$M_{\alpha}:D \to \mathcal{M}_{n}(\mathbb{R})$ be $\mathcal{C}^1$ quadratic matrix functions, let
$N_{\alpha}:D \to \mathcal{M}_{n,k}(\mathbb{R})$ be $\mathcal{C}^1$ rectangular matrix functions,
and let $u_{\alpha}:D \to\mathbb{R}^k=\mathcal{M}_{k,1}(\mathbb{R})$ be $\mathcal{C}^1$
vector functions, all indexed after $\alpha=\overline{1,m}$.

We consider the evolution PDE system
\begin{equation}
\label{aaa1}
\frac{\partial x}{\partial t^{\alpha}}
=
M_{\alpha}(t)x+N_{\alpha}(t)u_{\alpha}(t),
\quad
\forall \alpha=\overline{1,m}.
\end{equation}
Its complete integrability conditions are equivalent to

\begin{equation*}
\begin{split}
\frac{\partial M_{\alpha}}{\partial t^{\beta}}(t)
+
M_{\alpha}(t)M_{\beta}(t)
=
\frac{\partial M_{\beta}}{\partial t^{\alpha}}(t)
+
M_{\beta}(t)M_{\alpha}(t),
\end{split}
\end{equation*}
\begin{equation}
\label{aaa2}
\begin{split}
& M_{\alpha}(t)N_{\beta}(t)u_{\beta}(t)
+
\frac{\partial N_{\alpha}}{\partial t^{\beta}}(t)
u_{\alpha}(t)+
N_{\alpha}(t)
\frac{\partial u_{\alpha}}{\partial t^{\beta}}(t)\\
&=
M_{\beta}(t)N_{\alpha}(t)u_{\alpha}(t)
+
\frac{\partial N_{\beta}}{\partial t^{\alpha}}(t)
u_{\beta}(t)+
N_{\beta}(t)
\frac{\partial u_{\beta}}{\partial t^{\alpha}}(t),
\end{split}
\end{equation}
$\forall t\in D, \,\, \forall \alpha,\beta=\overline{1,m}$.

\begin{definition}
\label{definitia3.1} {\it Suppose that the matrix functions
$M_\alpha(\cdot)$ verify the relations} $(\ref{aa6})$,
$\forall t\in D, \,\, \forall \alpha,\beta=\overline{1,m}$.
{\it The vector space}
$$\mathcal{U} =
\Big \{u=(u_{\alpha})_{\alpha=\overline{1,m}}
\,\,
\Big |
\,
u_{\alpha}:D \to \mathbb{R}^k=\mathcal{M}_{k,1}(\mathbb{R}),\,
\mbox{\rm of class}\,\, \mathcal{C}^1, \forall \alpha=\overline{1,m}
$$
$$
\mbox{ \rm and which verify the relations $(\ref{aaa2})$ for all}\,\, \alpha, \beta
\Big \}
$$
{\it is called the control space}.
\end{definition}

From the Theorem \ref{teorema2.6}, we obtain immediately

\begin{theorem}
\label{teorema3.1} {\it If the matrix functions
$M_\alpha(\cdot)$ verify the relations} $(\ref{aa6})$,
$\forall t\in D, \,\, \forall \alpha,\beta=\overline{1,m}$
{\it and $u=(u_{\alpha})_{\alpha=\overline{1,m}}$
is a control, then the Cauchy problem}
\begin{equation*}
\frac{\partial x}{\partial t^{\alpha}}
=
M_{\alpha}(t)x+N_{\alpha}(t)u_{\alpha}(t),
\quad
\forall \alpha=\overline{1,m}.
\end{equation*}
\begin{equation*}
x(t_0)=x_0   \quad \quad (t_0\in D, \,\, x_0\in \mathbb{R}^n)
\end{equation*}
{\it has a unique solution
$$
x:D\to \mathbb{R}^n,\,\,x(t)=\chi(t,t_0)x_0
+
\int\limits_{\gamma_{t_0,t}}^{}
\chi(t,s)N_{\alpha}(s)u_{\alpha}(s)
\mbox{d}s^{\alpha},
$$
{\it where $\gamma_{t_0,t}$ is a piecewise $\mathcal{C}^1$ curve,
included in $D$, covered from $t_0$ to $t$}.

{\it The curvilinear integral $\displaystyle \int\limits_{\gamma}^{}
\chi(t,s)N_{\alpha}(s)u_{\alpha}(s))
\mbox{d}s^{\alpha}$ is path independent and the solution
$x(\cdot)$ is of class} $\mathcal{C}^2$}.
\end{theorem}

Further, in this paper, $D$ will be an open and convex subset of $\mathbb{R}^m$,
the $\mathcal{C}^1$ quadratic matrix functions
$M_{\alpha}:D \to \mathcal{M}_{n}(\mathbb{R})$, $\forall \alpha=\overline{1,m}$ will verify the relations (\ref{aa6}), $\forall t\in D, \,\, \forall \alpha,\beta=\overline{1,m}$
and the rectangular matrix functions $N_{\alpha}:D \to \mathcal{M}_{n,k}(\mathbb{R})$, will be of class
$\mathcal{C}^1$, $\forall \alpha=\overline{1,m}$.

\begin{definition}
\label{definitia3.2}
{\it The pair $(s,y),\,\,s\in D,\,\,y\in\mathbb{R}^n$
is called phase of the PDE system} $(\ref{aaa1})$.

$a)$ {\it Let $(t_0,x_0)$, $(s,y)$ $\in D \times \mathbb{R}^n$. We say that the phase
$(t_0,x_0)$ transfers to the phase $(s,y)$ if the Cauchy problems
$\{(\ref{aaa1}), x(t_0)=x_0 \}$ and} $\{(\ref{aaa1}), x(s)=y \}$
{\it have the same solution (for the same control $u(\cdot)$); or,
equivalently, the solution $x(t)$ of the Cauchy problem
} $\{(\ref{aaa1}), x(t_0)=x_0 \}$
{\it verifies also the condition $x(s)=y$. We will say that the control $u(\cdot)$
transfers the phase $(t_0,x_0)$ into the phase $(s,y)$}.

$b)$ {\it The phase $(t,x)$ is called reachable (respectively pseudo-reachable) if there exists a point
$t_0 \in D$, with $t_0^{\alpha}<t^{\alpha}$, $\forall \alpha$,
(respectively, if there exists a point $t_0 \in D$, $t_0\neq t$),
and there exists a control $u(\cdot)$ which transfers the phase $(t_0,0)$ into the phase $(t,x)$}.

$c)$ {\it The phase $(t,x)$ is called controllable (respectively, pseudo-controllable)
if there exists a point
$s\in D$, with $s^{\alpha}>t^{\alpha}$, $\forall \alpha$,
(respectively, if there exists a point $s \in D$, $s \neq t$),
and a control $u(\cdot)$ which transfers the phase $(t,x)$ into the phase $(s,0)$}.

$d)$ {\it Let $t_0,t \in D$, with $t_0^{\alpha}<t^{\alpha}$, $\forall \alpha$, (respectively, let $t_0,t \in D$, $t_0\neq t$)}.

{\it The PDE system $(\ref{aaa1})$ is called completely reachable
(respectively completely pseudo-reachable) from $t_0$ to $t$
if for any point $x \in \mathbb{R}^n$, the phase $(t_0,0)$ transfers to the phase $(t,x)$, i.e.,
for any $x$, the phase $(t,x)$ is reachable (respectively,  pseudo-reachable) with the same $t_0$}.

$e)$ {\it Let $t\in D$.} {\it The PDE system $(\ref{aaa1})$ is called completely reachable
(respectively, completely pseudo-reachable) at the moment $t$, if for any $t_0\in D$, with
$t_0^{\alpha}<t^{\alpha}$, $\forall \alpha$,
(respectively $\forall t_0\in D$, $t_0\neq t$),
and for any $x \in \mathbb{R}^n$, the phase $(t_0,0)$ transfers into the phase $(t,x)$}.

$f)$ {\it Let $t_0,t \in D$, with $t_0^{\alpha}<t^{\alpha}$, $\forall \alpha$ (respectively, let $t_0,t \in D$, $t_0\neq t$)}.

{\it The PDE system $(\ref{aaa1})$ is called completely controllable (respectively, completely pseudo-controllable) from $t_0$ to $t$
if for any point $x \in \mathbb{R}^n$, the phase $(t_0,x)$ transfers into the phase $(t,0)$, i.e.,
for any point $x$ the phase $(t_0,x)$ is controllable (respectively, pseudo-controllable) with the same $t$}.

$g)$ {\it Let $t_0\in D$. The PDE system $(\ref{aaa1})$ is called completely controllable
(respectively, completely pseudo-controllable) at the moment $t_0$, if
$\forall t\in D$, with $t^{\alpha}>t_0^{\alpha}$, $\forall \alpha$,
(respectively, $\forall t\in D$, $t_0\neq t$), and for any point $x \in \mathbb{R}^n$, the phase $(t_0,x)$
transfers into the phase $(t,0)$}.

$h)$ {\it The PDE system $(\ref{aaa1})$ is called completely reachable (respectively, completely pseudo-reachable) if it is
completely reachable (respectively, completely pseudo-reachable) at any moment of $D$.}

{\it The PDE system $(\ref{aaa1})$ is called completely controllable (respectively, completely pseudo-controllable) if it is completely
controllable (respectively, completely pseudo-controllable) at any moment of $D$.}
\end{definition}

The multitime control property does not only depend on the dimensions
$m$ and $n$ but on how matrices $M_\alpha$ and $N_\alpha$ interact.

The phase $(t_0,x_0)$ transfers into the phase $(t_1,y)$ $\Longleftrightarrow$
$\exists u(\cdot)$=($u_{\alpha}(\cdot)$) a control such that the solution
$x(\cdot)$ of the problem $\{(\ref{aaa1}), x(t_0)=x_0 \}$ verifies also $x(t_1)=y$, equivalent to
$$\exists u(\cdot)=(u_{\alpha}(\cdot))\,\,\mbox{ a control such that}$$
$$
x(t)=\chi(t,t_0)x_0 + \int\limits_{\gamma_{t_0,t}}^{}
\chi(t,t_0)\chi(t_0,s)N_{\alpha}(s)u_{\alpha}(s)\,\mbox{d}s^{\alpha}\,\,
\mbox{and}\,\, x(t_1)=y
$$
$$
\Longleftrightarrow
\exists
u(\cdot)=(u_{\alpha}(\cdot))\,\,\mbox{a control such that}
$$
$$
y=
\chi(t_1,t_0)
\Big(
x_0
+
\int\limits_{\gamma_{t_0,t_1}}^{}
\chi(t_0,s)N_{\alpha}(s)
u_{\alpha}(s)
\,\mbox{d}s^{\alpha}
\Big)
$$
$$
\Longleftrightarrow
\exists
u(\cdot)=(u_{\alpha}(\cdot))\,\,
\mbox{a control such that}
$$
$$
\chi(t_0,t_1)y-x_0
=
\int\limits_{\gamma_{t_0,t_1}}^{}
\chi(t_0,s)N_{\alpha}(s)
u_{\alpha}(s)
\,\mbox{d}s^{\alpha}.
$$

We introduce the set
\begin{equation*}
\mathcal{V}(t_0,t):=
\Bigg \{
\int\limits_{\gamma_{t_0,t}}^{}
\chi(t_0,s)N_{\alpha}(s)
u_{\alpha}(s)
\,\mbox{d}s^{\alpha}
\,
\Big|
\,
(u_{\alpha})_{\alpha =\overline{1,m}}\,
\mbox{ is a control}
\Bigg \}.
\end{equation*}
The set $\mathcal{V}(t_0,t)$ is a vector subspace of $\mathbb{R}^n$. It is called
{\it the controllability space}. Since the curvilinear integral is path independent, we remark that $\mathcal{V}(t_0,t)$
does not depend on the curve $\gamma_{t_0,t}$, which joins $t_0$ to $t$,
but depends on the multitimes $t_0$ and $t$.
Also $\chi(t,t_0)\mathcal{V}(t_0,t)=\mathcal{V}(t,t_0)$.

From the foregoing arguments, it follows immediately

\begin{theorem}
\label{teorema3.2} {\it Let us consider
the system $(\ref{aaa1})$, with the matrix functions
$M_\alpha(\cdot)$ verifying the relations} $(\ref{aa6})$.

{\it $i)$ The control $\displaystyle (u_{\alpha})_{\alpha =\overline{1,m}}$
transfers the phase $(t_0,x_0)$ to the phase $(t,y)$ if and only if}
\begin{equation*}
\chi(t_0,t)y-x_0
=
\int\limits_{\gamma_{t_0,t}}^{}
\chi(t_0,s)N_{\alpha}(s)
u_{\alpha}(s)
\,\mbox{\rm d}s^{\alpha}.
\end{equation*}

{\it $ii)$ The control $\displaystyle (u_{\alpha})_{\alpha =\overline{1,m}}$
transfers the phase $(t_0,x_0)$ to the phase $(t,0)$ if and only if}
$$
x_0
+
\int\limits_{\gamma_{t_0,t}}^{}
\chi(t_0,s)N_{\alpha}(s)
u_{\alpha}(s)
\,\mbox{\rm d}s^{\alpha}
=0.
$$

{\it $iii)$ The phase $(t_0,x_0)$ transfers into the phase $(t,y)$ if and only if
$$
x_0-\chi(t_0,t)y\in \mathcal{V}(t_0,t)
$$
equivalent to
$$
y-\chi(t,t_0)x_0 \in \mathcal{V}(t,t_0).
$$

$iv)$ The phase $(t_0,x_0)$ is controllable (respectively, pseudo-controllable) if and only if
$\exists t\in D$, with $t^{\alpha}> t^{\alpha}_0$,
$\forall \alpha$ (respectively, $\exists t\in D$, $t\neq t_0$) such that
$$
x_0\in \mathcal{V}(t_0,t).
$$

$v)$ The phase $(t,y)$ is reachable (respectively, pseudo-reachable) if and only if
$\exists t_0\in D$, with $t_0^{\alpha}< t^{\alpha}$,
$\forall \alpha$ (respectively, $\exists t_0\in D$, $t_0\neq t$) such that
$$
y\in \mathcal{V}(t,t_0).
$$

$vi)$ Let $t_0, t \in D$, with $t_0^{\alpha}<t^{\alpha}$, $\forall \alpha$
(respectively, let $t_0,t \in D$, $t_0\neq t$). The PDE system is completely controllable (respectively, completely pseudo-controllable)
from the multitime $t_0$ into the multitime $t$ if and only if
$$
\mathcal{V}(t_0,t)=\mathbb{R}^n,
\quad
\mbox{ equality equivalent to }\,\, \mathcal{V}(t,t_0)=\mathbb{R}^n.
$$

$vii)$ Let $t_0,t \in D$, with $t_0^{\alpha}<t^{\alpha}$, $\forall \alpha$ (respectively, let $t_0,t \in D$, $t_0\neq t$).
The PDE system is completely reachable (respectively, completely pseudo-reachable) from the multitime $t_0$ into the multitime $t$
if and only if
$$
\mathcal{V}(t,t_0)=\mathbb{R}^n,
\quad
\mbox{ equality equivalent to }\,\, \mathcal{V}(t_0,t)=\mathbb{R}^n.
$$

$viii)$ Let $t_0,t \in D$, with $t_0^{\alpha}<t^{\alpha}$, $\forall \alpha$ (respectively, let $t_0,t \in D$, $t_0\neq t$).
The PDE system is completely controllable (respectively, completely pseudo-controllable) from the multitime $t_0$ into the multitime $t$
if and only if it is completely reachable (respectively, completely pseudo-reachable)
from $t_0$ to $t$. }
\end{theorem}

According to the Theorem \ref{teorema2.4}, the curvilinear integral
$$
\int\limits_{\gamma}^{}
\chi(t_0,s)N_{\alpha}(s)
N_{\alpha}^\top(s)\chi(t_0,s)^\top
\mbox{d}s^{\alpha}
$$
is path independent if and only if, for any $\alpha,\beta=\overline{1,m}$, the following conditions are satisfied:
$$
\frac{\partial}{\partial s^{\beta}}
(\chi(t_0,s))N_{\alpha}
N_{\alpha}^\top\chi(t_0,s)^\top
+
\chi(t_0,s)
\frac{\partial N_{\alpha}}{\partial s^{\beta}}
N_{\alpha}^\top\chi(t_0,s)^\top
+
$$
$$
+
\chi(t_0,s)N_{\alpha}
\frac{\partial N_{\alpha}^\top}{\partial s^{\beta}}
\chi(t_0,s)^\top
+
\chi(t_0,s)N_{\alpha}N_{\alpha}^\top
\frac{\partial }{\partial s^{\beta}}
(\chi(t_0,s)^\top)
$$
$$
=
\frac{\partial}{\partial s^{\alpha}}
(\chi(t_0,s))N_{\beta}
N_{\beta}^\top\chi(t_0,s)^\top
+
\chi(t_0,s)
\frac{\partial N_{\beta}}{\partial s^{\alpha}}
N_{\beta}^\top\chi(t_0,s)^\top
+
$$
$$
+
\chi(t_0,s)N_{\beta}
\frac{\partial N_{\beta}^\top}{\partial s^{\alpha}}
\chi(t_0,s)^\top
+
\chi(t_0,s)N_{\beta}N_{\beta}^\top
\frac{\partial }{\partial s^{\alpha}}
(\chi(t_0,s)^\top)
$$
or
$$
-\chi(t_0,s)M_{\beta}N_{\alpha}
N_{\alpha}^\top\chi(t_0,s)^\top
+
\chi(t_0,s)
\frac{\partial N_{\alpha}}{\partial s^{\beta}}
N_{\alpha}^\top\chi(t_0,s)^\top
+
$$
$$
+
\chi(t_0,s)N_{\alpha}
\frac{\partial N_{\alpha}^\top}{\partial s^{\beta}}
\chi(t_0,s)^\top
-
\chi(t_0,s)N_{\alpha}N_{\alpha}^\top
M^\top_{\beta}
\chi(t_0,s)^\top
$$
$$
=
-\chi(t_0,s)M_{\alpha}N_{\beta}
N_{\beta}^\top\chi(t_0,s)^\top
+
\chi(t_0,s)
\frac{\partial N_{\beta}}{\partial s^{\alpha}}
N_{\beta}^\top\chi(t_0,s)^\top
+
$$
$$
+
\chi(t_0,s)N_{\beta}
\frac{\partial N_{\beta}^\top}{\partial s^{\alpha}}
\chi(t_0,s)^\top
-
\chi(t_0,s)N_{\beta}N_{\beta}^\top
M^\top_{\alpha}
\chi(t_0,s)^\top.
$$
Since the fundamental matrix $\chi(t_0,s)$ is invertible, the
foregoing equality is equivalent to
$$
-M_{\beta}N_{\alpha}
N_{\alpha}^\top
+
\frac{\partial N_{\alpha}}{\partial s_{\beta}}
N_{\alpha}^\top
+
N_{\alpha}
\frac{\partial N_{\alpha}^\top}{\partial s_{\beta}}
-
N_{\alpha}N_{\alpha}^\top
M^\top_{\beta}
$$
$$
=
-M_{\alpha}N_{\beta}
N_{\beta}^\top
+
\frac{\partial N_{\beta}}{\partial s_{\alpha}}
N_{\beta}^\top
+
N_{\beta}
\frac{\partial N_{\beta}^\top}{\partial s_{\alpha}}
-
N_{\beta}N_{\beta}^\top
M^\top_{\alpha}.
$$
In this way, we have proved

\begin{proposition}
\label{propozitia3.1} {\it Let $t_0 \in D$, fixed.
The curvilinear integral
$$
\int\limits_{\gamma }^{}
\chi(t_0,s)N_{\alpha}(s)
N_{\alpha}^\top(s)\chi(t_0,s)^\top
\mbox{\rm d}s^{\alpha}
$$
is path independent on $D$} ({\it in the sense of definition $\ref{definitia2.2}$})
{\it if and only if, for any $\alpha,\beta=\overline{1,m}$, the relations
\begin{equation}
\label{aaa4}
\begin{split}
M_{\alpha}N_{\beta}
N_{\beta}^\top
+
\frac{\partial N_{\alpha}}{\partial s^{\beta}}
N_{\alpha}^\top
+
N_{\alpha}
\frac{\partial N_{\alpha}^\top}{\partial s^{\beta}}
+
N_{\beta}N_{\beta}^\top
M^\top_{\alpha}
\\
=
M_{\beta}N_{\alpha}
N_{\alpha}^\top
+
\frac{\partial N_{\beta}}{\partial s^{\alpha}}
N_{\beta}^\top
+
N_{\beta}
\frac{\partial N_{\beta}^\top}{\partial s^{\alpha}}
+
N_{\alpha}N_{\alpha}^\top
M^\top_{\beta}
\end{split}
\end{equation}
are verified on $D$. This is equivalent to,
\begin{equation}
\label{aaa5}
\begin{split}
M_{\alpha}N_{\beta}
N_{\beta}^\top
+
\frac{\partial N_{\alpha}}{\partial s^{\beta}}
N_{\alpha}^\top
+
\Big (
M_{\alpha}N_{\beta}
N_{\beta}^\top
+
\frac{\partial N_{\alpha}}{\partial s^{\beta}}
N_{\alpha}^\top
\Big )^\top
\\
=
M_{\beta}N_{\alpha}
N_{\alpha}^\top
+
\frac{\partial N_{\beta}}{\partial s^{\alpha}}
N_{\beta}^\top
+
\Big(
M_{\beta}N_{\alpha}
N_{\alpha}^\top
+
\frac{\partial N_{\beta}}{\partial s^{\alpha}}
N_{\beta}^\top
\Big )^\top
\end{split}
\end{equation}
or
\begin{equation}
\label{aaa6}
\begin{split}
M_{\alpha}N_{\beta}
N_{\beta}^\top
+
N_{\alpha}
\frac{\partial N_{\alpha}^\top}{\partial s^{\beta}}
+
\Big (
M_{\alpha}N_{\beta}
N_{\beta}^\top
+
N_{\alpha}
\frac{\partial N_{\alpha}^\top}{\partial s^{\beta}}
\Big )^\top
\\
=
M_{\beta}N_{\alpha}
N_{\alpha}^\top
+
N_{\beta}
\frac{\partial N_{\beta}^\top}{\partial s^{\alpha}}
+
\Big(
M_{\beta}N_{\alpha}
N_{\alpha}^\top
+
N_{\beta}
\frac{\partial N_{\beta}^\top}{\partial s^{\alpha}}
\Big )^\top.
\end{split}
\end{equation}
It is sufficient, for example, that for any $\alpha,\beta=\overline{1,m}$, to have
\begin{equation}
\label{aaa7}
\begin{split}
M_{\alpha}N_{\beta}
N_{\beta}^\top
+
\frac{\partial N_{\alpha}}{\partial s^{\beta}}
N_{\alpha}^\top
=
M_{\beta}N_{\alpha}
N_{\alpha}^\top
+
\frac{\partial N_{\beta}}{\partial s^{\alpha}}
N_{\beta}^\top
\end{split}
\end{equation}
or }
\begin{equation}
\label{aaa8}
\begin{split}
M_{\alpha}N_{\beta}
N_{\beta}^\top
+
N_{\alpha}
\frac{\partial N_{\alpha}^\top}{\partial s^{\beta}}
=
M_{\beta}N_{\alpha}
N_{\alpha}^\top
+
N_{\beta}
\frac{\partial N_{\beta}^\top}{\partial s^{\alpha}}.
\end{split}
\end{equation}
\end{proposition}

\begin{proposition}
\label{propozitia3.2} {\it Let us suppose that the matrices
$M_\alpha(\cdot)$ verify the relations} $(\ref{aa6})$,
$\forall t\in D, \,\, \forall \alpha,\beta=\overline{1,m}$.
{\it We fix $t_0 \in D$. For each
$v \in \mathbb{R}^n$ and $\alpha=\overline{1,m}$,
we consider the functions}
\begin{equation*}
u_{\alpha,v}: D \to \mathbb{R}^k,
\quad
u_{\alpha,v}(s)
=
N_{\alpha}^\top(s)\chi(t_0,s)^\top v,
\quad
\forall
s \in D.
\end{equation*}

{\it The following statements are equivalent}

$i)$ {\it For any $v \in \mathbb{R}^n$, the family
$(u_{\alpha,v})_{\alpha=\overline{1,m}}$ is
a control for the PDE system} $(\ref{aaa1})$.

$ii)$ {\it For any $\alpha,\beta=\overline{1,m}$,
the relations} $(\ref{aaa4})$ {\it are satisfied on the set $D$, i.e.,}
\begin{equation*}
\begin{split}
M_{\alpha}N_{\beta}
N_{\beta}^\top
+
\frac{\partial N_{\alpha}}{\partial s^{\beta}}
N_{\alpha}^\top
+
N_{\alpha}
\frac{\partial N_{\alpha}^\top}{\partial s^{\beta}}
+
N_{\beta}N_{\beta}^\top
M^\top_{\alpha}
=
\\
=
M_{\beta}N_{\alpha}
N_{\alpha}^\top
+
\frac{\partial N_{\beta}}{\partial s^{\alpha}}
N_{\beta}^\top
+
N_{\beta}
\frac{\partial N_{\beta}^\top}{\partial s^{\alpha}}
+
N_{\alpha}N_{\alpha}^\top
M^\top_{\beta}.
\end{split}
\end{equation*}

$iii)$ {\it The curvilinear integral
$$
\int\limits_{\gamma}^{}
\chi(t_0,s)N_{\alpha}(s)
N_{\alpha}^\top(s)\chi(t_0,s)^\top
\mbox{\rm d}s^{\alpha}
$$
is path independent on the set $D$}.
\end{proposition}
\begin{proof} The family $\displaystyle (u_{\alpha,v})_{\alpha=\overline{1,m}}$
is a control if and only if it verifies,
$\forall \alpha,\beta=\overline{1,m}$, the relations $(\ref{aaa2})$ on the set $D$, i.e.,
$$
M_{\alpha}N_{\beta}
N_{\beta}^\top\chi(t_0,s)^\top v
+
\frac{\partial N_{\alpha}}{\partial s^{\beta}}
N_{\alpha}^\top\chi(t_0,s)^\top v
+
$$
$$
+
N_{\alpha}
\frac{\partial N_{\alpha}^\top}{\partial s^{\beta}}
\chi(t_0,s)^\top v
+
N_{\alpha}N^\top_{\alpha}
\frac{\partial}{\partial s^{\beta}}
(\chi(t_0,s)^\top)v
=
$$
$$
=
M_{\beta}N_{\alpha}
N_{\alpha}^\top\chi(t_0,s)^\top v
+
\frac{\partial N_{\beta}}{\partial s^{\alpha}}
N_{\beta}^\top\chi(t_0,s)^\top v
+
$$
$$
+
N_{\beta}
\frac{\partial N_{\beta}^\top}{\partial s^{\alpha}}
\chi(t_0,s)^\top v
+
N_{\beta}N^\top_{\beta}
\frac{\partial}{\partial s^{\alpha}}
(\chi(t_0,s)^\top)v
$$
equivalent to
$$
M_{\alpha}N_{\beta}
N_{\beta}^\top\chi(t_0,s)^\top v
+
\frac{\partial N_{\alpha}}{\partial s^{\beta}}
N_{\alpha}^\top\chi(t_0,s)^\top v
+
$$
$$
+
N_{\alpha}
\frac{\partial N_{\alpha}^\top}{\partial s^{\beta}}
\chi(t_0,s)^\top v
-
N_{\alpha}N^\top_{\alpha}
M^\top_{\beta}
\chi(t_0,s)^\top v
=
$$
$$
=
M_{\beta}N_{\alpha}
N_{\alpha}^\top\chi(t_0,s)^\top v
+
\frac{\partial N_{\beta}}{\partial s^{\alpha}}
N_{\beta}^\top\chi(t_0,s)^\top v
+
$$
$$
+
N_{\beta}
\frac{\partial N_{\beta}^\top}{\partial s^{\alpha}}
\chi(t_0,s)^\top v
-
N_{\beta}N^\top_{\beta}
M^\top_{\alpha}
\chi(t_0,s)^\top v
$$
or
$$
\Big (
M_{\alpha}N_{\beta}
N_{\beta}^\top
+
\frac{\partial N_{\alpha}}{\partial s^{\beta}}
N_{\alpha}^\top
+
N_{\alpha}
\frac{\partial N_{\alpha}^\top}{\partial s^{\beta}}
-
N_{\alpha}N^\top_{\alpha}
M^\top_{\beta}
\Big )
\chi(t_0,s)^\top v
=
$$
$$
=
\Big (
M_{\beta}N_{\alpha}
N_{\alpha}^\top
+
\frac{\partial N_{\beta}}{\partial s^{\alpha}}
N_{\beta}^\top
+
N_{\beta}
\frac{\partial N_{\beta}^\top}{\partial s^{\alpha}}
-
N_{\beta}N^\top_{\beta}
M^\top_{\alpha}
\Big )
\chi(t_0,s)^\top v,
$$
equivalent to
\begin{equation*}
\begin{split}
\Big (
M_{\alpha}N_{\beta}
N_{\beta}^\top
+
\frac{\partial N_{\alpha}}{\partial s^{\beta}}
N_{\alpha}^\top
+
N_{\alpha}
\frac{\partial N_{\alpha}^\top}{\partial s^{\beta}}
+
N_{\beta}N_{\beta}^\top
M^\top_{\alpha}
\Big )
\chi(t_0,s)^\top v
\\
=
\Big (
M_{\beta}N_{\alpha}
N_{\alpha}^\top
+
\frac{\partial N_{\beta}}{\partial s^{\alpha}}
N_{\beta}^\top
+
N_{\beta}
\frac{\partial N_{\beta}^\top}{\partial s^{\alpha}}
+
N_{\alpha}N_{\alpha}^\top
M^\top_{\beta}
\Big )
\chi(t_0,s)^\top v.
\end{split}
\tag{$\ast$}
\end{equation*}
The implication $ii) \Longrightarrow i)$ follows immediately.

Let us prove $i) \Longrightarrow ii)$.
Since for each $v\in \mathbb{R}^n$, the family
$\displaystyle (u_{\alpha,v})_{\alpha=\overline{1,m}}$ is a control, it follows
that the relations $(\ast)$ hold for any $v\in \mathbb{R}^n$, whence
we deduce that for any matrix
$A \in \mathcal{M}_{n,p}(\mathbb{R})$, $\forall p$, we have
\begin{equation*}
\begin{split}
\Big (
M_{\alpha}N_{\beta}
N_{\beta}^\top
+
\frac{\partial N_{\alpha}}{\partial s^{\beta}}
N_{\alpha}^\top
+
N_{\alpha}
\frac{\partial N_{\alpha}^\top}{\partial s^{\beta}}
+
N_{\beta}N_{\beta}^\top
M^\top_{\alpha}
\Big )
\chi(t_0,s)^\top A
\\
=
\Big (
M_{\beta}N_{\alpha}
N_{\alpha}^\top
+
\frac{\partial N_{\beta}}{\partial s^{\alpha}}
N_{\beta}^\top
+
N_{\beta}
\frac{\partial N_{\beta}^\top}{\partial s^{\alpha}}
+
N_{\alpha}N_{\alpha}^\top
M^\top_{\beta}
\Big )
\chi(t_0,s)^\top A.
\end{split}
\end{equation*}
Taking $A=I_n$, we find
\begin{equation*}
\begin{split}
\Big (
M_{\alpha}N_{\beta}
N_{\beta}^\top
+
\frac{\partial N_{\alpha}}{\partial s^{\beta}}
N_{\alpha}^\top
+
N_{\alpha}
\frac{\partial N_{\alpha}^\top}{\partial s^{\beta}}
+
N_{\beta}N_{\beta}^\top
M^\top_{\alpha}
\Big )
\chi(t_0,s)^\top
\\
=
\Big (
M_{\beta}N_{\alpha}
N_{\alpha}^\top
+
\frac{\partial N_{\beta}}{\partial s^{\alpha}}
N_{\beta}^\top
+
N_{\beta}
\frac{\partial N_{\beta}^\top}{\partial s^{\alpha}}
+
N_{\alpha}N_{\alpha}^\top
M^\top_{\beta}
\Big )
\chi(t_0,s)^\top .
\end{split}
\end{equation*}
The matrix $\chi(t_0,s)^\top$ is invertible. The last equality
is multiplied in the right-hand side by $(\chi(t_0,s)^\top)^{-1}$,
obtaining the relation (\ref{aaa4}).

The equivalence of the statements $ii)$ and
$iii)$ is just the Proposition \ref{propozitia3.1}.
\end{proof}

\begin{definition}
\label{definitia3.3} {\it Suppose that, for any
$\alpha,\beta=\overline{1,m}$, the relations $(\ref{aaa4})$ are true.
The matrix function}
$$ \mathcal{C}:D \times D\to \mathcal{M}_n(\mathbb{R}),\,\,
\mathcal{C}(t_0,t):=
\int\limits_{\gamma_{t_0,t}}^{}
\chi(t_0,s)N_{\alpha}(s)
N_{\alpha}^\top(s)\chi(t_0,s)^\top
\mbox{d}s^{\alpha}
$$
{\it is called the controllability gramian}.

{\it  The matrix function
$$
\mathcal{R}(t_0,t):=
\int\limits_{\gamma_{t_0,t}}^{}
\chi(t,s)N_{\alpha}(s)
N_{\alpha}^\top(s)\chi(t,s)^\top
\mbox{d}s^{\alpha}
$$ is called the reachability gramian.}
\end{definition}

The controllability gramian is used to determine whether or not a linear PDE system is controllable.
The reachability gramian is used to determine whether or not a linear PDE system is reachable.
One observes immediately that
$$\mathcal{R}(t_0,t)=- \,\,\mathcal{C}(t,t_0), \quad \forall t_0,t\in D$$
and
$$\chi(t,t_0)\mathcal{C}(t_0,t)\chi(t,t_0)^\top
=-\,\,\mathcal{C}(t,t_0), \quad \forall t_0,t\in D .$$
Hence the matrices $\mathcal{C}(t_0,t)$, $\mathcal{C}(t,t_0)$,
$\mathcal{R}(t_0,t)$, $\mathcal{R}(t,t_0)$
have all the same rank.

\begin{definition}
\label{definitia3.4} {\it Let $A\in \mathcal{M}_{p,q}(\mathbb{R})$
be a real matrix. Denote $\mbox{\rm Im}(A)$ and $\mbox{\rm Ker}(A)$,
the image, respectively the kernel of the linear map
$$
f: \mathbb{R}^q= \mathcal{M}_{q,1}(\mathbb{R}) \to \mathbb{R}^p= \mathcal{M}_{p,1}(\mathbb{R}),\,\,f(x)=Ax.
$$
Of course, the subset $\mbox{\rm Im}(A)$ is a vector subspace
of $\mathcal{M}_{p,1}(\mathbb{R})$ generated by the columns of the matrix} $A$.
\end{definition}

\begin{theorem}
\label{teorema3.3} {\it In the conditions of
Theorem $\ref{teorema3.1}$, if
$\forall \alpha,\beta=\overline{1,m}$, the conditions $(\ref{aaa4})$
are true, then, for any $t$ and $t_0$ with
$t^{\alpha}\geq t^{\alpha}_0$,
$\forall\alpha=\overline{1,m}$ (or
$t^{\alpha}\leq t^{\alpha}_0$,
$\forall\alpha=\overline{1,m}$), we have}
$$
\mathcal{V}(t_0,t)=\mbox{\rm Im}(\mathcal{C}(t_0,t)).
$$
\end{theorem}
\begin{proof} The inclusion
$\mathcal{V}(t_0,t)\subseteq \mbox{\rm Im}(\mathcal{C}(t_0,t))$
is equivalent to
$$
(\mathcal{V}(t_0,t))^\bot
\supseteq
(\mbox{\rm Im}(\mathcal{C}(t_0,t)))^\bot
=
\mbox{\rm Ker} ((\mathcal{C}(t_0,t))^\top).
$$
We have $b\in \mbox{\rm Ker} ((\mathcal{C}(t_0,t))^\top)
\Longleftrightarrow
((\mathcal{C}(t_0,t))^\top)b=0
\Longleftrightarrow
b^\top \mathcal{C}(t_0,t)=0.$ Hence $b^\top \mathcal{C}(t_0,t)b=0$.

The controllability gramian is independent on the curve $\gamma$ covered
from the multitime $t_0$
to the multitime $t$. Particularly, we fix $\gamma$ as being the straight line segment which joins
the points $t_0$, $t$, i.e.,
$\gamma(\tau)=\tau(t-t_0)+t_0$, $\tau\in [0,1]$.
It follows
$$
\mathcal{C}(t_0,t)=
\int\limits_{0}^{1}
\sum_{\alpha=1}^m
(t^{\alpha}-t^{\alpha}_0)
\chi(t_0,\gamma(\tau))
N_{\alpha}(\gamma(\tau))
N_{\alpha}^\top(\gamma(\tau))
\chi(t_0,\gamma(\tau))^\top
\mbox{d}\tau
$$
$$
=
\int\limits_{0}^{1}
\sum_{\alpha \mbox{ with } t^{\alpha}\neq t^{\alpha}_0}
(t^{\alpha}-t^{\alpha}_0)
\chi(t_0,\gamma(\tau))
N_{\alpha}(\gamma(\tau))
N_{\alpha}^\top(\gamma(\tau))
\chi(t_0,\gamma(\tau))^\top
\mbox{d}\tau
$$
On the other hand we get
$$
\int\limits_{0}^{1}
\sum_{\alpha \mbox{ with } t^{\alpha}>t^{\alpha}_0}
(t^{\alpha}-t^{\alpha}_0)
b^\top
\chi(t_0,\gamma(\tau))
N_{\alpha}(\gamma(\tau))
N_{\alpha}^\top(\gamma(\tau))
\chi(t_0,\gamma(\tau))^\top
b
\,\mbox{d}\tau
=0,
$$
or
$$
\int\limits_{0}^{1}
\sum_{\alpha \mbox{ with } t^{\alpha}>t^{\alpha}_0}
(t^{\alpha}-t^{\alpha}_0)
\Big \|
b^\top
\chi(t_0,\gamma(\tau))
N_{\alpha}(\gamma(\tau))
\Big \|^2
\mbox{d}\tau
=0.
$$
It follows that for any $\alpha$ with $t^{\alpha}\neq t^{\alpha}_0$ and, $\forall \tau \in [0,1]$, we have
$$
b^\top
\chi(t_0,\gamma(\tau))
N_{\alpha}(\gamma(\tau))
=0.
$$
Hence
$$
b^\top
\int\limits_{0}^{1}
\sum_{\alpha=1}^m
\chi(t_0,\gamma(\tau))
N_{\alpha}(\gamma(\tau))
u_{\alpha}(\gamma(\tau))
\dot{\gamma^{\alpha}}(\tau)
\mbox{d}\tau
$$
$$
= \int\limits_{0}^{1}
\sum_{\alpha \mbox{ with } t^{\alpha}\neq t^{\alpha}_0}
b^\top
\chi(t_0,\gamma(\tau))
N_{\alpha}(\gamma(\tau))
u_{\alpha}(\gamma(\tau))
(t^{\alpha}-t^{\alpha}_0)
\mbox{d}\tau
=0.
$$
Consequently $b\in (\mathcal{V}(t_0,t))^\bot$.

Let us prove the inclusion $\mbox{\rm Im}(\mathcal{C}(t_0,t))\subseteq
\mathcal{V}(t_0,t)$. For that, we select
$$w=\mathcal{C}(t_0,t)v\in \mbox{\rm Im}(\mathcal{C}(t_0,t)),\,\,
w= \int\limits_{\gamma_{t_0,t}}^{}
\chi(t_0,s)N_{\alpha}(s)
N_{\alpha}^\top(s)\chi(t_0,s)^\top
v
\,\mbox{d}s^{\alpha}.
$$
We choose
$$
u_{\alpha}(s)
=N_{\alpha}^\top(s)\chi(t_0,s)^\top v.
$$
According to the Proposition \ref{propozitia3.2}, the family
$\displaystyle (u_{\alpha})_{\alpha=\overline{1,m}}$ is
a control, hence
$$
w
=
\int\limits_{\gamma_{t_0,t}}^{}
\chi(t_0,s)N_{\alpha}(s)
u_{\alpha}(s)
\,\mbox{d}s^{\alpha}
\in \mathcal{V}(t_0,t).
$$
\end{proof}

If, for any $\alpha,\beta=\overline{1,m}$, the relations (\ref{aaa4}) are true
(hence we have also $\mathcal{V}(t_0,t)=\mbox{\rm Im}(\mathcal{C}(t_0,t))$
for $t^{\alpha}>(or <)t_0^{\alpha}$, $\forall \alpha$),
then from the Theorem \ref{teorema3.2} it follows

\begin{theorem}
\label{teorema3.4} {\it Suppose that we are in the conditions
of the Theorem $\ref{teorema3.1}$,
and furthermore, for any $\alpha,\beta=\overline{1,m}$, the conditions $(\ref{aaa4})$ are true}.

{\it $i)$ Let $t_0,t \in D$, $t_0\neq t$,
$t_0^{\alpha}\leq t^{\alpha}$,
$\forall \alpha$ or
$t^{\alpha}\leq t_0^{\alpha}$,
$\forall \alpha$.
Then the phase $(t_0,x_0)$ transfers to the phase $(t,y)$ if and only if}
$$
x_0-\chi(t_0,t)y\in \mbox{\rm Im}(\mathcal{C}(t_0,t)),
$$
{\it which is equivalent to}
$$
y - \chi(t,t_0)x_0 \in \mbox{\rm Im}(\mathcal{C}(t,t_0)).
$$

$ii)$ {\it The phase $(t_0,x_0)$ is controllable if and only if
$\exists t\in D$, with $t^{\alpha}> t^{\alpha}_0$, $\forall \alpha$, such that}
$$
x_0\in \mbox{\rm Im}(\mathcal{C}(t_0,t)).
$$

$iii)$ {\it The phase $(t,y)$ is reachable if and only if
$\exists t_0\in D$, with $t_0^{\alpha}< t^{\alpha}$, $\forall \alpha$, such that}
$$
y\in \mbox{\rm Im}(\mathcal{C}(t,t_0)).
$$

$iv)$ {\it Let $t_0,t \in D$, with $t_0^{\alpha}<t^{\alpha}$, $\forall \alpha$.
The PDE system is completely controllable from $t_0$ to $t$ if and only if}
$$
\mbox{\rm rank }\mathcal{C}(t_0,t)=n
\quad
\Big(
\Longleftrightarrow
\mathbb{R}^n=\mbox{\rm Im}(\mathcal{C}(t_0,t))
\,\Big).
$$

$v)$ {\it Let $t_0,t \in D$, with $t_0^{\alpha}<t^{\alpha}$, $\forall \alpha$.
The PDE system is completely reachable from $t_0$ into $t$ if and only if}
$$
\mbox{\rm rank }\mathcal{C}(t_0,t)=n
\quad
\Big(
\Longleftrightarrow
\mathbb{R}^n=\mbox{\rm Im}(\mathcal{C}(t_0,t))\,\Big).
$$
\end{theorem}

\section{Controlled autonomous linear PDE \\system of first order}

A very special case is those of controlled autonomous
linear PDE system of first order, when
the matrix functions $M_{\alpha}$, $N_{\alpha}$ are
constants. Then the relation (\ref{aa6}) becomes
\begin{equation*}
M_{\alpha}M_{\beta}=M_{\beta}M_{\alpha},
\end{equation*}
the relation (\ref{aaa2}) reduces to
\begin{equation*}
M_{\alpha}N_{\beta}u_{\beta}
+
N_{\alpha}
\frac{\partial u_{\alpha}}{\partial s^{\beta}}
=M_{\beta}N_{\alpha}u_{\alpha}
+
N_{\beta}
\frac{\partial u_{\beta}}{\partial s^{\alpha}}
\end{equation*}
and the relation (\ref{aaa4}) can be written as
\begin{equation*}
M_{\alpha}N_{\beta}N^\top_{\beta}
+
N_{\beta}N^\top_{\beta}M^\top_{\alpha}
=
M_{\beta}N_{\alpha}N^\top_{\alpha}
+
N_{\alpha}N^\top_{\alpha}M^\top_{\beta}.
\end{equation*}

On the other hand, we have
\begin{equation*}
\chi(t,t_0)
=
e^{
M_{\alpha}(t^{\alpha}-t^{\alpha}_0)
}.
\end{equation*}
In this case, the fundamental matrix function $\chi(t,t_0)$
is defined for any $(t,t_0) \in \mathbb{R}^m \times \mathbb{R}^m$,
and not only for $(t,t_0) \in D \times D$. Moreover, $\chi(t,t_0)=\chi(t-t_0,0)$.

The controllability gramian matrix becomes
\begin{equation*}
\mathcal{C}(t_0,t)=
\int\limits_{\gamma_{t_0,t}}^{}
e^{
-
M_{\beta}
(s^{\beta}-t^{\beta}_0)
}
N_{\alpha}
N_{\alpha}^\top
\Big(
e^{
-
M_{\beta}
(s^{\beta}-t^{\beta}_0)
}
\Big )
^\top
\mbox{d}s^{\alpha}.
\end{equation*}
Since the relations (\ref{aaa4}) are verified on $\mathbb{R}^m$,
from the Proposition \ref{propozitia3.1}
it follows that the curvilinear integral is path independent on
$\mathbb{R}^m$. Hence the controllability gramian  $\mathcal{C}(t_0,t)$
is defined on $\mathbb{R}^m \times \mathbb{R}^m$ and also $\displaystyle\mathcal{C}(t_0,t)=\mathcal{C}(0,t-t_0)$.

In fact, if the matrix functions $M_{\alpha}$ and $N_{\alpha}$ are constant
$\forall \alpha$, one can take $D=\mathbb{R}^m$.

\begin{definition}
\label{definitia4.1} {\it For $\alpha=\overline{1,m}$,  let us consider
the constant matrices $M_{\alpha}\in \mathcal{M}_n(\mathbb{R})$,
$N_{\alpha}\in \mathcal{M}_{n,k}(\mathbb{R})$, such that}
\begin{equation*}
M_{\alpha}M_{\beta}
=
M_{\beta}M_{\alpha},
\qquad
\forall
\alpha, \beta=\overline{1,m}.
\end{equation*}
{\it For each $\alpha=\overline{1,m}$, we define the matrix
\begin{equation*}
G_{\alpha}= \Big( N_{\alpha} \,\,\, M_1N_{\alpha} \,\,\,
M_2N_{\alpha}\,\,\, \dots
\,\,\, M_mN_{\alpha}\,\,\, \dots \,\,\,
M_1^{k_{1}} \cdot M_2^{k_{2}}
\cdot \ldots \cdot M_m^{k_{m}} \cdot
N_{\alpha} \,\,\, \dots \,\,\, \Big)
\end{equation*}
made from all block matrices of the form
$$
M_1^{k_{1}} \cdot M_2^{k_{2}} \cdot \ldots
\cdot M_m^{k_{m}}\cdot
N_{\alpha}
$$
with \, $0\leq k_{1};k_{2}; \ldots;k_{m}\leq{n-1}$.
Further, we need to specify the order in which one ranges the block matrices
$M_1^{k_{1}} \cdot M_2^{k_{2}} \cdot \ldots
\cdot M_m^{k_{m}}\cdot
N_{\alpha}$ in the matrix $G_{\alpha}$.
In this way, the matrix $G_{\alpha}$ will be well defined}.

{\it On the set
\begin{equation*}
\Big\{
(k_1;k_2; \dots ; k_m)
\in
\mathbb{N}^m
\Bigm |
0\leq k_{1};k_{2}; \ldots;k_{m}\leq{n-1}
\Big\},
\end{equation*}
we define the following order relation, denoted by $\preceq$}:
\begin{equation*}
(k_1;k_2; \dots ; k_m)
\preceq
(q_1;q_2; \dots ; q_m)
\quad \mbox { {\it if} }
\end{equation*}
\begin{equation*}
k_1+k_2+ \dots + k_m
<
q_1+q_2+ \dots + q_m
\end{equation*}
\begin{equation*}
\quad \mbox { {\it or} }
\end{equation*}
\begin{equation*}
k_1+k_2+ \dots + k_m
=
q_1+q_2+ \dots + q_m
\quad \mbox { {\it and} }
\end{equation*}
{\it $k_1>q_1$ or $\exists j=\overline{2,m}$
such that $k_1=q_1$,
$k_2=q_2$, $\dots$, $k_{j-1}=q_{j-1}$, $k_{j}>q_{j}$ }
\begin{equation*}
\quad \mbox { {\it or} }
\end{equation*}
\begin{equation*}
(k_1;k_2; \dots ; k_m)
=
(q_1;q_2; \dots ; q_m).
\end{equation*}
{\it One verifies quickly that $\preceq$ is an order relation. The
block matrices}
$$
M_1^{k_{1}} \cdot M_2^{k_{2}} \cdots M_m^{k_{m}}\cdot
N_{\alpha}
$$
{\it are ranged in $G_{\alpha}$ in increasing order of
$(k_1;k_2; \dots ; k_m)$, relative to the order relation $\preceq$. This
means in fact that the block matrices are written in the increasing order of the sum
$k_1+k_2+ \dots + k_m$; in case that two such sums are equal,
the block matrices are written in lexicographic decreasing order of
$(k_1;k_2; \dots ; k_m)$.

The matrix
\begin{equation*}
G:= \Big( \,\, G_1 \quad G_2 \quad
\dots \quad G_m \,\, \Big).
\end{equation*}
is called the controllability matrix of the PDE system} $(\ref{aaa1})$.
\end{definition}

\begin{theorem}
\label{teorema4.1} {\it Let $t_0,t \in \mathbb{R}^m$
such that
$t^{\alpha}_0<t^{\alpha}$,
$\forall \alpha=\overline{1,m}$. Then}
\begin{equation*}
\mbox{\rm rank }\mathcal{C}(t,t_0)=
\mbox{\rm rank }\mathcal{C}(t_0,t)=\mbox{\rm rank }G,
\end{equation*}
{\it relation equivalent to}
\begin{equation*}
\mbox{\rm Im }\mathcal{C}(t_0,t)=\mbox{\rm Im }G,
\end{equation*}
{\it or}
\begin{equation*}
\mbox{\rm Ker} (\mathcal{C}(t_0,t)^\top)=
\mbox{\rm Ker}  (G^\top).
\end{equation*}
\end{theorem}

\begin{proof} We already have seen that
$\mbox{\rm rank }\mathcal{C}(t,t_0)= \mbox{\rm rank }\mathcal{C}(t_0,t)$.
Hence, it is enough to show the equality
$\displaystyle \mbox{\rm Ker} (\mathcal{C}(t_0,t)^\top)= \mbox{\rm Ker} (G^\top)$.

The inclusion $\displaystyle
\mbox{\rm Ker} (\mathcal{C}(t_0,t)^\top)
\subseteq
\mbox{\rm Ker}  (G^\top)$: let $b$
such that $(\mathcal{C}(t_0,t))^\top b=0$,
or $b^\top \mathcal{C}(t_0,t) =0$. Consequently we have
and $b^\top \mathcal{C}(t_0,t)b =0$.

Let $s$ such as $t^{\alpha}_0<s^{\alpha}<t^{\alpha}$, $\forall
\alpha=\overline{1,m}$, $s$ fixed but arbitrarily chosen. Instead of
the curve $\gamma_{t_0,t}$ we select the union of the segments
$[t_0,s]$ and $[s,t]$, covered from $t_0$ to $s$, respectively from
$s$ to $t$. Let us parameterize these segments by
$$
\lambda_1(\tau)
=
(1-\tau)t_0+\tau s,
\quad \forall \tau \in [0,1],
$$
$$
\lambda_2(\tau)
=
(1-\tau)s+\tau t,
\quad \forall \tau \in [0,1].
$$
Obviously $\gamma_{t_0,t}$ is a piecewise $\mathcal{C}^1$ curve.
From $b^\top \mathcal{C}(t_0,t)b =0$, we find
\begin{equation*}
\begin{split}
\int\limits_{0}^{1}
b^\top
e^{
\displaystyle
-M_{\beta}
(\lambda_1^{\beta}(\tau)-t^{\beta}_0)
}
N_{\alpha}
N_{\alpha}^\top
\Big(
e^{
\displaystyle
-M_{\beta}
(\lambda_1^{\beta}(\tau)-t^{\beta}_0)
}
\Big ) ^\top
b\,
(s^{\alpha}-t_0^{\alpha})
\,\mbox{d}\tau
\end{split}
\end{equation*}
\begin{equation*}
\begin{split}
+
\int\limits_{0}^{1}
b^\top
e^{
\displaystyle
-M_{\beta}
(\lambda_2^{\beta}(\tau)-t^{\beta}_0)
}
N_{\alpha}
N_{\alpha}^\top
\Big(
e^{
\displaystyle
-M_{\beta}
(\lambda_2^{\beta}(\tau)-t^{\beta}_0)
}
\Big ) ^\top
b\,
(t^{\alpha}-s^{\alpha})
\,\mbox{d}\tau = 0;
\end{split}
\end{equation*}
\begin{equation*}
\int\limits_{0}^{1}
\left(
\Big \|
b^\top
e^{
\displaystyle
-M_{\beta}
(\lambda_1^{\beta}(\tau)-t^{\beta}_0)
}
N_{\alpha}
\Big \| ^2
(s^{\alpha}-t_0^{\alpha})
\right.
\end{equation*}
\begin{equation*}
+
\left.
\Big \|
b^\top
e^{
\displaystyle
-M_{\beta}
(\lambda_2^{\beta}(\tau)-t^{\beta}_0)
}
N_{\alpha}
\Big \| ^2
(t^{\alpha}-s^{\alpha})
\right)
\mbox{d}\tau
=0.
\end{equation*}
Since the integrand is positive, it follows
\begin{equation*}
\sum_{\alpha=1}^m
\Big \|
b^\top
e^{
\displaystyle
-M_{\beta}
(\lambda_1^{\beta}(\tau)-t^{\beta}_0)
}
N_{\alpha}
\Big \| ^2
(s^{\alpha}-t_0^{\alpha})
+
\end{equation*}
\begin{equation*}
+
\sum_{\alpha=1}^m
\Big \|
b^\top
e^{
\displaystyle
-M_{\beta}
(\lambda_2^{\beta}(\tau)-t^{\beta}_0)
}
N_{\alpha}
\Big \| ^2
(t^{\alpha}-s^{\alpha})
=0, \,\,
\forall \tau \in [0;1].
\end{equation*}
But $s^{\alpha}-t_0^{\alpha}>0$ \, and $t^{\alpha}-s^{\alpha}>0$, hence
\begin{equation*}
b^\top
e^{
\displaystyle
-M_{\beta}
(\lambda_1^{\beta}(\tau)-t^{\beta}_0)
}
N_{\alpha}
=
0, \,\,
\forall \tau \in [0;1], \,\,
\forall \alpha=\overline{1,m}
\end{equation*}
and
\begin{equation*}
b^\top
e^{
\displaystyle
-M_{\beta}
(\lambda_2^{\beta}(\tau)-t^{\beta}_0)
}
N_{\alpha}
=0, \,\,
\forall \tau \in [0;1], \,\,
\forall \alpha=\overline{1,m}.
\end{equation*}
We set $\tau=1$, hence $\lambda_1^{\beta}(1)=s^{\beta}$ and one obtains
\begin{equation*}
b^\top
e^{
\displaystyle
-M_{\beta}
(s^{\beta}-t^{\beta}_0)
}
N_{\alpha}
=0, \,\,
\forall \alpha=\overline{1,m};
\end{equation*}
valid equalities
$\forall\, s^{\beta}\in (t^{\beta}_0;t^{\beta})$,
\, $\beta=\overline{1,m}$. But, taking into account the continuity, it follows
\begin{equation}
\label{aaaa1}
b^\top
e^{
\displaystyle
-M_{\beta}
(s^{\beta}-t^{\beta}_0)
}
N_{\alpha}
=0, \,\,
\forall \alpha=\overline{1,m},
\end{equation}
\begin{equation*}
\mbox{and}\,\,\forall \, s
\mbox{ such that }
t^{\beta}_0
\leq
s^{\beta}
<
t^{\beta}, \,\, \forall \beta.
\end{equation*}
Let\, $0\leq k_1,k_2, \ldots , k_m \leq n-1$.
In (\ref{aaaa1}) we differentiate with respect to
$s^1$ of $k_1$ times, with respect to
$s^2$ of $k_2$ times, $\dots$, with respect to
$s^m$ of $k_m$ times (we take also into account the relation (\ref{aa6})):
\begin{equation*}
(-1)^{k_1+k_2+\ldots +k_m}
b^\top
e^{
\displaystyle
-M_{\beta}
(s^{\beta}-t^{\beta}_0)
}
M_1^{k_1}M_2^{k_2}
\cdot
\ldots
\cdot
M_m^{k_m}
N_{\alpha}
=0, \,\,
\end{equation*}
while for $s=t_0$ we have
\begin{equation*}
b^\top
M_1^{k_1}M_2^{k_2}
\cdots
M_m^{k_m}
N_{\alpha}
=0, \,\,
\end{equation*}
hence also $\displaystyle b^\top G=0$ or
$\displaystyle G^\top b=0$, i.e., $b\in \mbox{\rm Ker} (G^\top)$.

The inclusion $\displaystyle
\mbox{\rm Ker} (G^\top)
\subseteq
\mbox{\rm Ker}(\mathcal{C}(t_0,t)^\top)$: let $b$ such that
$\displaystyle G^\top b=0$, or
$\displaystyle b^\top G=0$. It follows that
\begin{equation*}
b^\top
M_1^{k_1}M_2^{k_2}
\cdots
M_m^{k_m}
N_{\alpha}
=0,
\quad
\forall
\,0\leq k_1,k_2, \ldots , k_m \leq n-1.
\end{equation*}
From the Hamilton-Cayley Theorem, and taking into account the relations
$M_{\alpha}M_{\beta}=M_{\beta}M_{\alpha}$,
we deduce
\begin{equation*}
b^\top
M_1^{k_1}M_2^{k_2}
\cdots
M_m^{k_m}
N_{\alpha}
=0,
\quad
\forall\,
k_1,k_2, \ldots , k_m \geq0,
\end{equation*}
\begin{equation*}
b^\top
(t^{1}_0-s^1)^{k_1}
M_1^{k_1}
(t^{2}_0-s^2)^{k_2}
M_2^{k_2}
\cdots
(t^{m}_0-s^m)^{k_m}
M_m^{k_m}
N_{\alpha}
=0,
\end{equation*}
valid equalities
$\forall\,
k_1,k_2, \ldots , k_m \geq0$. Hence we have $\forall p\geq0$, $\forall \alpha=\overline{1,m}$,
\begin{equation*}
b^\top
\Big(
\sum_{\beta=1}^m
M_{\beta}
(t^{\beta}_0-s^{\beta})
\Big)^p
N_{\alpha}
=0.
\end{equation*}
\begin{equation*}
b^\top
e^{
\displaystyle
M_{\beta}
(t^{\beta}_0-s^{\beta})
}
N_{\alpha}
=
\sum_{p=0}^{\infty}
\frac{1}{p\,!}
b^\top
\Big(
M_{\beta}
(t^{\beta}_0-s^{\beta})
\Big)^p
N_{\alpha}
=0.
\end{equation*}
It follows that $b^\top C(t_0,t)=0$
or $(\mathcal{C}(t_0,t))^\top b=0$, hence
$b\in \mbox{\rm Ker}(\mathcal{C}(t_0,t)^\top)$.
\end{proof}

\begin{remark} From the proof of the Theorem $\ref{teorema4.1}$,
we see that the inclusion
$$\displaystyle
\mbox{\rm Ker} (G^\top)
\subseteq
\mbox{\rm Ker}(\mathcal{C}(t_0,t)^\top),
\,\,
\mbox { i.e.}, \,\,
\mbox{\rm Im } \mathcal{C}(t_0,t)
\subseteq
\mbox{\rm Im } G,
$$
is true for any $t_0,t \in \mathbb{R}^m$, hence we have also
$$
\mbox{\rm rank} \,\mathcal{C}(t_0,t)\leq \mbox{\rm rank}\, G,
\qquad
\forall
t_0,t \in \mathbb{R}^m.
$$
\end{remark}

\begin{example} There exist linear autonomous PDE systems for which we have $t_0,t\in \mathbb{R}^m$,
$t_0\neq t$, $t_0^{\alpha} \leq t^{\alpha}$,
$\forall \alpha$, such that
$\mbox{\rm rank }\displaystyle \mathcal{C}(t_0;t) < \mbox{\rm rank }G$. Indeed,
let us take
$$
m=2,\quad n=2, \quad k=1, \quad D=\mathbb{R}^2
$$
$$\displaystyle
N_1=\left(
       \begin{array}{c}
         1 \\
         0 \\
       \end{array}
     \right)
,\quad
N_2=\left(
       \begin{array}{c}
         0 \\
         1 \\
       \end{array}
     \right)
,\quad
M_1=
\left(
  \begin{array}{cc}
    1 & 0 \\
    0 & 0 \\
  \end{array}
\right)
,\quad
M_2=0.
$$
\end{example}
These verify $M_1M_2=M_2M_1=0$ and $M_1N_1=N_1,\,\, M_1N_2=0.$

Control space: $u=(u_1,u_2)$ is a control if and only if
$$
M_1N_2u_2+N_1\frac{\partial u_1}{\partial t_2} = M_2N_1u_1+N_2\frac{\partial u_2}{\partial t_1};
$$
the relations $M_1N_2=0$, $M_2=0$  transforms the foregoing PDE in
$$
N_1\frac{\partial u_1}{\partial t_2} = N_2\frac{\partial u_2}{\partial t_1}
$$
equivalent to
$$
\frac{\partial u_1}{\partial t_2}
=0;
\quad \quad
\frac{\partial u_2}{\partial t_1}
=0.
$$

Consequently $u=(u_1,u_2)$ is a control if and only if there exist $f_1,f_2:\mathbb{R}\to \mathbb{R}$, of
class $\mathcal{C}^1$ such that
$$
u_1(t^1,t^2)=f_1(t^1),
\quad \quad
u_2(t^1,t^2)=f_2(t^2),
\quad \quad
\forall (t^1,t^2)\in \mathbb{R}^2.
$$

The relation
$$
M_1N_2N_2^\top+N_2N_2^\top M_1^\top = M_2N_1N_1^\top+N_1N_1^\top M_2^\top
$$
(see the condition $(\ref{aaa4})$) is obvious since $M_1N_2=0$ and $M_2=0$.

The rank of the matrix
$$
G=
\Big(\,\,
N_1 \quad M_1N_1 \quad M_2N_1 \quad M_1M_2N_1
\quad
N_2 \quad M_1N_2 \quad M_2N_2 \quad M_1M_2N_2
\,\,
\Big)
$$
$$
=
\left(
  \begin{array}{cccccccc}
    1 & 1 & 0 & 0 & 0 & 0 & 0 & 0 \\
    0 & 0 & 0 & 0 & 1 & 0 & 0 & 0 \\
  \end{array}
\right)
$$
is $2$.

We compute the matrix $\mathcal{C}(t_0,t)$, with
$t_0=0=(0,0)$, $t=(t^1,0)$, ($t^1\neq 0$), i.e.,
$t_0^2=t^2=0$. For that, we select $\gamma (\tau)=(\tau,0)$,
$\tau \in [0,t^1]$; $\gamma$ is a curve joining the two-time $(0,0)$ with $t=(t^1,0)$. Then
$$
\mathcal{C}((0,0);(t^1,0))
=
\displaystyle
\int\limits_{0}^{t^1}
\displaystyle
e^
{\displaystyle
-\tau M_1
}\cdot
N_1N_1^\top
\cdot
\left(
e^
{\displaystyle
-\tau M_1
}
\right)
^\top
\mbox{d} \tau.
$$
But
$$
e^
{\displaystyle
-\tau M_1}
=
\left(
  \begin{array}{cc}
    e^
{\displaystyle
-\tau} & 0 \\
    0 & 0 \\
  \end{array}
\right),
\quad \quad
e^
{\displaystyle
-\tau M_1
}\cdot
N_1
=
\left(
  \begin{array}{c}
     e^{\displaystyle -\tau} \\
    0 \\
  \end{array}
\right);
$$
$$
\mathcal{C}((0,0);(t^1,0))
=
\displaystyle
\int\limits_{0}^{t^1}
\displaystyle
\left(
  \begin{array}{c}
     e^{\displaystyle -\tau} \\
    0 \\
  \end{array}
\right)
\left(
  \begin{array}{cc}
    e^{\displaystyle -\tau} & 0 \\
  \end{array}
\right)
\mbox{d} \tau
=
\displaystyle
\int\limits_{0}^{t^1}
\displaystyle
\left(
  \begin{array}{cc}
    e^{\displaystyle -2\tau} & 0 \\
    0 & 0 \\
  \end{array}
\right )
\mbox{d} \tau;
$$
$$
\mathcal{C}((0,0);(t^1,0))
=
\displaystyle
\left(
  \begin{array}{cc}
     \frac{ \displaystyle 1- e^{-2t^1}}{ 2} & 0 \\
    0 & 0 \\
  \end{array}
\right).
$$
Consequently the rank of the matrix $\mathcal{C}((0,0);(t^1,0))$ is $1$,
strictly smaller than the rank of $G$.

The Theorem \ref{teorema3.4} can be rewritten as

\begin{theorem}
\label{teorema4.2} {\it For
$\alpha=\overline{1,m}$,  let us consider the constant matrices
$M_{\alpha}\in \mathcal{M}_n(\mathbb{R})$,
$N_{\alpha}\in \mathcal{M}_{n,k}(\mathbb{R})$,
such that }
\begin{equation*}
M_{\alpha}M_{\beta}
=
M_{\beta}M_{\alpha},
\qquad
\forall
\alpha, \beta=\overline{1,m},
\end{equation*}
\begin{equation*}
M_{\alpha}N_{\beta}N^\top_{\beta}
+
N_{\beta}N^\top_{\beta}M^\top_{\alpha}
=
M_{\beta}N_{\alpha}N^\top_{\alpha}
+
N_{\alpha}N^\top_{\alpha}M^\top_{\beta},
\qquad
\forall
\alpha, \beta=\overline{1,m}.
\end{equation*}

{\it We consider the autonomous PDE system}

\begin{equation*}
\frac{\partial x}{\partial t^{\alpha}}
=
M_{\alpha}x+N_{\alpha}u_{\alpha}(t),
\quad
\forall \alpha=\overline{1,m},
\end{equation*}
{\it where $u=(u_{\alpha})_{\alpha=\overline{1,m}}$
is a control, i.e.,
$u_{\alpha}:\mathbb{R}^m \to \mathbb{R}^k
=
\mathcal{M}_{k,1}(\mathbb{R})$ is of class $\mathcal{C}^1$, $\forall\alpha$, and}
\begin{equation*}
M_{\alpha}N_{\beta}u_{\beta}(t)
+
N_{\alpha}
\frac{\partial u_{\alpha}}{\partial s^{\beta}}(t)
=M_{\beta}N_{\alpha}u_{\alpha}(t)
+
N_{\beta}
\frac{\partial u_{\beta}}{\partial s^{\alpha}}(t),
\,\,\,
\forall t \in
\mathbb{R}^m,
\,\,
\forall
\alpha,\beta
=\overline{1,m}.
\end{equation*}

{\it Let $G$ be the controllability matrix of this PDE system}.

{\it $i)$ If $t^{\alpha}\,>\,(or <)\,t_0^{\alpha}$, $\forall \alpha$,
then the phase $(t_0,x_0)$ transfers to the phase $(t,y)$ if and only if }
$$
x_0-
\displaystyle e^{\displaystyle
M_{\alpha}(t^{\alpha}_0-t^{\alpha})
}
y\in \mbox{\rm Im}(G),
$$
{\it equivalent to }
$$
y-\
\displaystyle e^{\displaystyle
M_{\alpha}(t^{\alpha}-t_0^{\alpha})}
x_0 \in \mbox{\rm Im}(G).
$$

$ii)$ {\it The phase $(t_0,x_0)$ is controllable if and only if }
$
x_0\in \mbox{\rm Im}(G).
$\\
{\it One observes that if exists a multitime $t_0$
such that the phase $(t_0,x_0)$ is controllable, then for any multitime $t$,
the phases $(t,x_0)$ are controllable.

$iii)$ The phase $(t,y)$ is reachable if and only if }
$
y\in \mbox{\rm Im}(G).
$\\
{\it One observes that if there exists a multitime $t$ such that the phase $(t,y)$ is
reachable, then for any multitime $s$, the phases $(s,y)$ are reachable}.

$iv)$ {\it If the phase $(t_0,x_0)$ is controllable (or reachable),
then for any multitime $t$, the phases $(t,x_0)$ are controllable and reachable}.

$v)$ {\it Let $t_0,t \in D$, with $t_0^{\alpha}<t^{\alpha}$, $\forall \alpha$.
The PDE system is completely controllable from
$t_0$ to $t$ if and only if}
$
\mbox{\rm rank }G=n.
$

$vi)$ {\it Let $t_0,t \in D$, with $t_0^{\alpha}<t^{\alpha}$, $\forall \alpha$.
The PDE system is completely reachable from $t_0$ to $t$ if and only if}
$
\mbox{\rm rank }G=n.
$

$vii)$ {\it If there exist $t_0,t \in D$, with $t_0^{\alpha}<t^{\alpha}$, $\forall \alpha$ and
if the PDE system is completely controllable (or completely reachable) from $t_0$ to $t$, then
the PDE system is completely controllable and completely reachable (equivalent to}
$\mbox{\rm rank }G=n $).
\end{theorem}

\begin{example} Let us give an example of multitime linear PDE system,
with constant matrices $M_{\alpha}$, $N_{\alpha}$, for which
we have $\mbox{\rm rank }G=n$,
but no state $(t_0,x_0)$, with $x_0\neq 0$, is controllable. Let us consider
$$
m=3;\,\,n=3;\,\,k=1;\,\, D=\mathbb{R}^3;
$$
\begin{equation*}
M_1=M_2=M_3
{\stackrel {\mathrm {not }}{=}}
M
=
\left(
  \begin{array}{ccc}
    0 & 0 & 1 \\
    1 & 0 & 0 \\
    0 & 1 & 0 \\
  \end{array}
\right)
\end{equation*}
\begin{equation*}
N_1=
\left(
  \begin{array}{c}
    1 \\
    0 \\
    0 \\
  \end{array}
\right),
\quad
N_2=
\left(
  \begin{array}{c}
    0 \\
    1 \\
    0 \\
  \end{array}
\right),
\quad
N_3=
\left(
  \begin{array}{c}
    0 \\
    0 \\
    1 \\
  \end{array}
\right).
\end{equation*}
\end{example}
The foregoing matrices verify the relations
$$
MN_1=N_2, \,
MN_2=N_3, \,
MN_3=N_1
$$
and the relations $(\ref{aa6})$ are obvious.

Let us determine the control space: $u=(u_1,u_2,u_3)$ is a control
if and only if the relations $(\ref{aaa2})$ are true.
If in $(\ref{aaa2})$ we put $\alpha=2$, $\beta=3$, then
\begin{equation*}
M_2N_3u_3(t)+N_2\frac{\partial u_2 }{\partial t^3}(t)
=
M_3N_2u_2(t)+N_3\frac{\partial u_3 }{\partial t^2}(t),
\quad
\forall t\in \mathbb{R}^3,
\end{equation*}
or
\begin{equation*}
u_3(t)N_1+\frac{\partial u_2 }{\partial t^3}(t)N_2
=
u_2(t)N_3+\frac{\partial u_3 }{\partial t^2}(t)N_3,
\quad
\forall t\in \mathbb{R}^3.
\end{equation*}
Since $N_1$, $N_2$, $N_3$ are linearly independent, it follows that
$$
u_3(t)=0,\,\,u_2(t)+\frac{\partial u_3 }{\partial t^2}(t)=0,\,\,\forall t\in \mathbb{R}^3,
$$
hence $u_3(t)=u_2(t)=0$, $\forall t\in \mathbb{R}^3$.

If in $(\ref{aaa2})$ we take $\alpha=1$, $\beta=2$, then we get
\begin{equation*}
M_1N_2u_2(t)+N_1\frac{\partial u_1 }{\partial t^2}(t)
=
M_2N_1u_1(t)+N_2\frac{\partial u_2 }{\partial t^1}(t),
\quad
\forall t\in \mathbb{R}^3,
\end{equation*}
or
\begin{equation*}
N_1\frac{\partial u_1 }{\partial t^2}(t)
=
N_2u_1(t),
\quad
\forall t\in \mathbb{R}^3,
\end{equation*}
hence $u_1(t)=0$, $\forall t\in \mathbb{R}^3$. We have proved
that the control space is $\{0\}$, and consequently
$$
\mathcal{V}(t_0,t)=\{0\},
\quad
\forall (t_0,t)
\in
\mathbb{R}^3
\times
\mathbb{R}^3.
$$
These results and the Theorem \ref{teorema3.2}
show that no phase $(t_0,x_0)$, with $x_0\neq 0$,
is controllable or reachable.

We remark that the controllability matrix $G$ has the rank $3=n$. The conclusion
of the Theorem \ref{teorema4.2} is no longer valid. The reason is that
the conditions $(\ref{aaa4})$ are not true.

\section{Comments}

In this paper, the functions which define the PDE
systems (for example (\ref{aa1}), (\ref{aaa1}), etc.) are of class
$\mathcal{C}^1$ and satisfy the complete integrability conditions
(of type (\ref{aa2}), (\ref{aa6})+(\ref{aa7}), (\ref{aa6})+(\ref{aaa2})). Also,
the general relations (\ref{aa4}) are verified - we have linear
PDE systems. Hence, throughout,
the Cauchy problem $\{(\ref{aaa1}), \,\, x(t_0)=x_0\}$
has a unique solution, global defined, and it is of class
$\mathcal{C}^2$ (Theorems \ref{teorema2.1}, \ref{teorema2.2}, \ref{teorema2.3}).

Sometimes, in the papers \cite{12} -- \cite{23}, the functions which define
the PDE systems (for example, the controls) are
piecewise $\mathcal{C}^1$ functions; in this case,
the complete integrability conditions are piecewise satisfied.
Identically, the solutions will verify the PDEs in the piecewise
sense. Generally, the solutions are not continuous functions.

In the paper \cite{13} it is indicated a construction of the solution of a
Cauchy problem associated to a linear PDE system. It is similar (in
a certain sense) with those obtained in the Theorem \ref{teorema3.1}. But, in
this context, the Cauchy problem has not a unique solution. To
maintain this idea, we give the following example: let
$m=2$, $n=1$ (hence $x(\cdot)=x_1(\cdot)$), $k=1$,
$D=\mathbb{R}^2$,
$$
M_1(t)=M_2(t)=0,
\quad
N_1(t)=N_2(t)=1,
\quad
\forall
t \in \mathbb{R}^2;
$$
\begin{equation*}
u_1(t^1,t^2)
=
\begin{cases}
\displaystyle 1, \mbox{ if } t^1+t^2 \geq 1 \\
\displaystyle 0, \mbox{ if } t^1+t^2 < 1
\end{cases};
\quad
u_2(t^1,t^2)=0,
\forall
(t^1,t^2) \in \mathbb{R}^2.
\end{equation*}
Set $t_0=(0,0)$, $x_0=0$, i.e., $x(0,0)=0$ and we formulate the Cauchy problem
$$
\displaystyle
\frac{\partial x}{\partial t^1} = u_1(t^1,t^2) \vspace{0.2 cm},\,\, \displaystyle\frac{\partial x}{\partial t^2}=0,\,\, x(0,0)=0.
$$
This PDE system satisfies the piecewise complete integrability conditions
(\ref{aa6}) and (\ref{aaa2}) (this can be easily checked);
the conditions (\ref{aaa2}) are true there where the
control $u_1(\cdot)$ is of class $\mathcal{C}^1$,
i.e., on the non-connected set
$$
\mathbb{R}^2
\setminus
\{(t^1,t^2) | t^1+t^2 =1\}.
$$

Here, we have $\chi (t,s)=1$, $\forall (t,s)\in
\mathbb{R}^2 \times \mathbb{R}^2$. If
$t=(t^1,t^2)\in \mathbb{R}^2$ is an arbitrary fixed two-time, then
$$
x(t^1,t^2)=\chi(t,t_0)x_0 + \int\limits_{\gamma_{t_0,t}}^{}\chi(t,s)N_{\alpha}(s)u_{\alpha}(s)\mbox{d}s^{\alpha} = \int\limits_{\gamma_{t_0,t}}^{}u_{1}(s^1, s^2)\mbox{d}s^{1}.
$$

For $t^1+t^2<1$, we obtain obviously $x(t^1,t^2)=0$. It remains to study the case $t^1+t^2>1$.

Let $a\in \mathbb{R}$ be an arbitrary point. We consider the curve $\gamma_{t_0,t}$
consisting in two segments: the first is the segment
which joins the point $t_0=(0,0)$ to the point $(a,1-a)$, on the straight line $s^1+s^2=1$,
where $u_1(\cdot)$ is discontinuous;
the second segment joins the point $(a,1-a)$ to the point $(t^1,t^2)$. The first segment,
without the point $(a,1-a)$, is situated in the semiplane
$
\{(s^1,s^2) | s^1+s^2 <1\},
$
and the second, without the point $(a,1-a)$, is included in the semiplane
$
\{(s^1,s^2) | s^1+s^2 >1 \}.
$
A parametrization of the second segment is
$$
s^1 (\tau)
=(1-\tau) a + \tau t^1,
\,\,
s^2 (\tau)
=(1-\tau) (1-a) + \tau t^2,
\quad
\tau \in [0,1].
$$
Taking into account that $u_1(\cdot)$ vanishes on the first segment, we find
$$
x(t^1,t^2)=
\int\limits_{0}^{1}
1
\cdot
(t^1-a)
\mbox{d}\tau
=t^1-a.
$$
Consequently, the solution is given by the formula
\begin{equation*}
x(t^1,t^2)
=
\begin{cases}
\displaystyle t^1-a, \mbox{ if } t^1+t^2 > 1 \\
\displaystyle \quad 0 \quad , \mbox{ if } t^1+t^2 < 1,
\end{cases}
\end{equation*}
on the set
$$
\mathbb{R}^2
\setminus
\{(t^1,t^2) | t^1+t^2 =1\}
$$
and $x(0,0)=0$. Here we recognize an infinity of solutions since $a$ is an arbitrary point.

The foregoing solution $x(\cdot)$ can be extended to a continuous function at $(a,1-a)$,
but in rest the function $x(\cdot)$
is discontinuous on the straight line $\{(t^1,t^2) | t^1+t^2 =1\}$
(for any given values on this straight line).

\center


\begin{thebibliography}{}

\bibitem{1} L. Cesari, \emph{Existence theorems for abstract
multidimensional control problems.}
Journal of Optimization Theory and Applications 1970; {\bf 6}(3), pp. 210 -- 236.

\bibitem{2} L. C. Evans, \emph{An Introduction to Mathematical Optimal Control Theory}.
Lecture Notes, University of California, Department of Mathematics, Berkeley, 2005,
http://math.berkeley.edu/$\sim$evans/control.course.pdf.

\bibitem{3} C. Ghiu, \emph{Popov-Belevich-Hautus theorem for linear
multitime autonomous dynamical systems.}
Scientific Bulletin, Series A 2010; {\bf 72}(4), pp. 93 -- 106.

\bibitem{4} T. Malakorn, {\it Multidimensional linear systems
and robust control}, PhD Thesis, Virginia Polytechnic Institute
and State University, 2003,
http://math.ucsd.edu/$\sim$toal/ETD.pdf.

\bibitem{5} \c St. Miric\u a, {\it Differential and Integral Equations}
(in Romanian), vol. II, Editorial House of University of Bucharest, 1999.

\bibitem{6} P. Pardalos, V. Yatsenko, {\it Optimization and
Control of Bilinear Systems}, Springer, 2009.

\bibitem{7} S. Pickenhain, M. Wagner, \emph{Pontryagin
principle for state-constrained control problems
governed by a first-order PDE system.}
Journal of Optimization Theory and Applications 2000; {\bf 107}(2), pp. 297 -- 330.

\bibitem{8} S. Pickenhain, M. Wagner, \emph{Piecewise
continuous controls in Dieudonn\'{e}-Rachevski type
problems.}
Journal of Optimization Theory and Applications 2005; {\bf 127}(1), pp. 145 -- 163.

\bibitem{9} V. Prepeli\c t\u a, T. Vasilache, M. Doroftei, {\it Control Theory},
University Politehnica of Bucharest, 1997.

\bibitem{10} V. Prepeli\c t\u a, \emph{Criteria of reachability
for 2D continous-discrete systems.}
Rev. Roumaine Math. Pures Appl. 2003; {\bf 48}(1), pp. 81 -- 93.

\bibitem{11} V. Prepeli\c t\u a, \emph{Minimal realization algorithm
for (q,r)-D hybrid systems.}
WSEAS Transactions on Systems 2009; {\bf 8}(1), pp. 22 -- 33.

\bibitem{12} C. Udri\c ste, \emph{Controllability and observability
of multitime linear PDE systems.} Proceedings of
The Sixth Congress of Romanian Mathematicians, Bucharest,
Romania 2007; {\bf 1}, pp. 313 -- 319.

\bibitem{13} C. Udri\c ste, \emph{Multitime controllability,
observability and bang-bang principle.}
Journal of Optimization Theory and Applications 2008; {\bf 139}(1), pp. 141 -- 157.

\bibitem{14} C. Udri\c ste, \emph{Simplified multitime maximum principle.}
Balkan Journal of Geometry and Its Applications 2009; {\bf 14}(1), pp. 102 -- 119.

\bibitem{15} C. Udri\c ste, \emph{Nonholonomic approach of
multitime maximum principle.} Balkan Journal of Geometry and
Its Applications 2009; {\bf 14}(2), pp. 111 -- 126.

\bibitem{16} C. Udri\c ste, I. \c Tevy, \emph{Multitime
linear-quadratic regulator problem based on curvilinear integral.}
Balkan Journal of Geometry and Its Applications 2009; {\bf 14}(2), pp. 127 -- 137.

\bibitem{17} C. Udri\c ste, I. \c Tevy, \emph{Multitime dynamic
programming for curvilinear integral actions.} Journal of
Optimization Theory and Applications 2010; {\bf 146}, pp. 189 -- 207.

\bibitem{18} C. Udri\c ste, \emph{Equivalence of multitime
optimal control problems.} Balkan Journal of Geometry
and Its Applications 2010; {\bf 15}(1), pp. 155 -- 162.

\bibitem{19} C. Udri\c ste and I. \c Tevy, {\em Multitime dynamic programming for curvilinear integral actions},
J. Optim. Theory Appl. 2010; {\bf 146}, pp. 189--207.


\bibitem{20} { C. Udri\c ste, I. \c{T}evy}, {\em Multitime dynamic programming for multiple integral actions}, J. Glob. Optim. 2011; {\bf 51, 2}, pp. 345-360.  

\bibitem{21} { C. Udri\c ste}, {\em Multitime Maximum Principle for Curvilinear Integral Cost}, 
Balkan J. Geom. Appl. 2011; {\bf 16, 1}, pp. 128-149.

\bibitem{22} { C. Udri\c ste, A. Bejenaru}, {\em Multitime optimal control with area integral costs on boundary}, 
Balkan J. Geom. Appl. 2011; {\bf 16, 2}, pp. 138-154.

\bibitem{23} Constantin Udri\c ste, Multitime maximum principle approach of
minimal submanifolds and harmonic maps, arXiv:1110.4745v1 [math.DG] 21 Oct 2011.


\bibitem{24} M. Wagner, \emph{Pontryagin's maximum
principle for Dieudonn\'{e}-Rashevsky type problems
involving Lipschitz functions.} Optimization 1999; {\bf 46}, pp. 165 -- 184.

\end{thebibliography}
\end{document}